\documentclass[11pt,draftcls,onecolumn]{IEEEtran}
\usepackage{epsfig} 
\usepackage{times} 
\usepackage{amsmath,amsthm} 
\usepackage{amssymb}  
\usepackage{enumerate}
\usepackage{cite}
\usepackage{color}
\usepackage{xspace}
\usepackage{algorithmic}
\usepackage[ruled]{algorithm}
\usepackage{subfigure}
\graphicspath{{./figures/}}
\usepackage{graphicx}
\usepackage{epstopdf}
\usepackage{float}
\DeclareGraphicsExtensions{.png}

\newtheorem{lemma}{Lemma}
\newtheorem{theorem}{Theorem}

\newtheorem{rem}{Remark}
\newtheorem{assumption}{Assumption}
\newtheorem{proposition}{Proposition}
\newtheorem{corollary}{Corollary}
\newtheorem{example}{Example}

\newcommand{\xw}[1]{\color{black}{{#1}}\color{black}\xspace}

\begin{document}

\title{\LARGE \bf Fenchel Dual Gradient Methods for Distributed Convex Optimization over Time-varying Networks}

\author{Xuyang Wu and Jie Lu\thanks{X. Wu and J. Lu are with the School of Information Science and Technology, ShanghaiTech University, 201210 Shanghai, China. Email: {\tt \{wuxy, lujie\}@shanghaitech.edu.cn}. }
\thanks{This work has been supported by the National Natural Science Foundation of China under grant 61603254, the Shanghai Pujiang Program under grant 16PJ1406400, and the Natural Science Foundation of Shanghai under grant 16ZR1422500.}}
\maketitle

\begin{abstract}
In the large collection of existing distributed algorithms for convex multi-agent optimization, only a handful of them provide convergence rate guarantees on agent networks with time-varying topologies, which, however, restrict the problem to be unconstrained. Motivated by this, we develop a family of distributed Fenchel dual gradient methods for solving constrained, strongly convex but not necessarily smooth multi-agent optimization problems over time-varying undirected networks. The proposed algorithms are constructed based on the application of weighted gradient methods to the Fenchel dual of the multi-agent optimization problem, and can be implemented in a fully decentralized fashion. We show that the proposed algorithms drive all the agents to both primal and dual optimality asymptotically under a minimal connectivity condition and at sublinear rates under a standard connectivity condition. Finally, the competent convergence performance of the distributed Fenchel dual gradient methods is demonstrated via simulations.
\end{abstract}

%
%
\section{Introduction}\label{sec:introduction}

In many engineering scenarios, a network of agents often need to jointly make a decision so that a global cost consisting of their local costs is minimized and certain global constraints are satisfied. Such a multi-agent optimization problem has found a considerable number of applications, such as estimation by sensor networks \cite{Rabbat04}, network resource allocation \cite{Beck14}, and cooperative control \cite{Giselsson13}.

To address convex multi-agent optimization in an efficient, robust, and scalable way, distributed optimization algorithms have been substantially exploited, which allow each agent to reach an optimal or suboptimal decision by repeatedly exchanging its own information with neighbors \cite{Rabbat04,Nedic10,Lee13,LinP16,QuG17b,Jakovetic14,ShiW15,QuG17,XiC17,Nedic15,Nedic16b,Nedic17,Johansson06,Boyd11,Koshal11,Duchi12,ZhuM12,Giselsson13,Patrinos13,Beck14,ChangTH14,Necoara14,Bianchi16,Margellos16,Johansson09,Ram09,Wei13a,Wei13b,Varagnolo16,LuJ12,Kia15,LouY16}. One typical approach is to let the agents perform consensus operations so as to mix their decisions that are updated using first-order information of their local objectives (e.g., \cite{Nedic10,Lee13,Jakovetic14,Nedic15,ShiW15,LinP16,QuG17b,Nedic16b,Nedic17,QuG17,XiC17}). Recently, rates of convergence to optimality have been established for a few consensus-based algorithms. By assuming that the problem is unconstrained and smooth (i.e., the gradient of each local objective is Lipschitz) and that the network is fixed, the consensus-based multi-step gradient methods \cite{Jakovetic14,ShiW15,QuG17,XiC17} are able to achieve sublinear rates of convergence, and also linear rates if the local objectives are further (restricted) strongly convex. Unlike these algorithms, the Subgradient-Push method \cite{Nedic15}, the Gradient-Push method \cite{Nedic16b}, the DIGing algorithm \cite{Nedic17}, and the Push-DIGing algorithm \cite{Nedic17} can be implemented over time-varying networks and still provide convergence rate guarantees. Specifically, Subgradient-Push converges to optimality at a sublinear rate of $O(\ln k/\sqrt{k})$ for unconstrained, nonsmooth problems with bounded subgradients \cite{Nedic15}. In addition, when the problem is unconstrained, strongly convex, and smooth, an $O(\ln k/k)$ rate is established for Gradient-Push \cite{Nedic16b}, and linear rates are provided for DIGing and Push-DIGing \cite{Nedic17}.

Another standard approach is to utilize dual decomposition techniques, which often lead to a dual problem with a decomposable structure, so that it can be solved in a distributed fashion by classic optimization methods including the gradient projection method, the accelerated gradient methods, the method of multipliers, and their variants (e.g., \cite{Johansson06,Boyd11,Koshal11,Duchi12,ZhuM12,Giselsson13,Patrinos13,Beck14,ChangTH14,Necoara14,Bianchi16,Margellos16}). Compared with the aforementioned consensus-based primal methods, many distributed dual/primal-dual algorithms can handle more complicated coupling constraints, yet still manage to achieve sublinear rates of convergence to dual and primal optimality when the dual function is smooth, and achieve linear rates when the dual function is also strongly concave. Despite this advantage, most of such methods require a fixed network topology. Although the primal-dual subgradient methods in \cite{ZhuM12}, the primal-dual perturbation method in \cite{ChangTH14}, and the proximal-minimization-based method in \cite{Margellos16} cope with time-varying agent networks, they only guarantee asymptotic convergence to optimality and no results on convergence rate are provided. In addition to the above two approaches, there are other lines of research on distributed optimization, including incremental optimization methods (e.g., \cite{Rabbat04,Johansson09,Ram09}), 
distributed Newton methods (e.g., \cite{Wei13a,Wei13b,Varagnolo16}), 
and continuous-time distributed optimization algorithms (e.g., \cite{LuJ12,Kia15,LouY16}).

This paper is motivated by the lack of distributed optimization algorithms in the literature that are able to address \emph{constrained} convex multi-agent optimization at a guaranteed \emph{convergence rate} over \emph{time-varying} networks. We propose, in this paper, a family of distributed Fenchel dual gradient methods that are able to solve a class of constrained multi-agent optimization problems at sublinear rates on time-varying \xw{undirected} networks, where the local objectives of the agents are strongly convex but not necessarily differentiable and the global constraint is the intersection of the local convex constraints of the agents. 

To develop such algorithms, we first derive the Fenchel dual of the multi-agent optimization problem, which consists of a separable, smooth dual function and a coupling linear constraint. Additionally, the gradient of the Fenchel dual function can be evaluated in parallel by the agents. We then utilize a class of weighted gradient methods to solve the Fenchel dual problem, which can be implemented over time-varying networks in a distributed fashion and can be viewed as a generalization of the distributed weighted gradient methods in \cite{XiaoL06b,Lakshmanan08}. We show that the proposed Fenchel dual gradient algorithms asymptotically converge to both dual and primal optimality if the agents and their infinitely occurring interactions form a connected graph. We also show that the dual optimality is reached at an $O(1/k)$ rate and the primal optimality is achieved at an $O(1/\sqrt{k})$ rate if the underlying agent interaction graph during every $B$ iterations is connected. Finally, the efficacy of the Fenchel dual gradient methods is illustrated through numerical examples.

The outline of the paper is as follows: Section~\ref{sec:probform} formulates the multi-agent optimization problem, and Section~\ref{sec:dualalgo} develops the distributed Fenchel dual gradient methods. Section~\ref{sec:convanal} establishes the convergence results of the proposed algorithms. Section~\ref{sec:numericalexample} presents simulation results, and Section~\ref{sec:conclusion} concludes the paper. All the proofs are included in the appendix. \xw{This paper is a significantly improved version of an earlier, $6$-page conference paper \cite{Wu17b}.}

Throughout the paper, we use $\|\cdot\|$ to represent the Euclidean norm and $\|\cdot\|_1$ the $\ell_1$ norm. For any set $X\subseteq\mathbb{R}^d$, $\operatorname{int}X$ represents its interior and $|X|$ its cardinality. Let $P_X(x)\!=\!\operatorname{arg\;min}_{y\in X}\|x-y\|$ denote the projection of $x\in\mathbb{R}^d$ onto $X$, which uniquely exists if $X$ is closed and convex. The ball centered at $x\in\mathbb{R}^d$ with radius $r>0$ is denoted by $B(x,r):=\{y\in\mathbb{R}^d:\|y-x\|\le r\}$. The floor of a real number is represented by $\lfloor \cdot\rfloor$. For any $\mathbf{x}\in\mathbb{R}^{nd}$, $\mathbf{x}=(x_1^T,\ldots,x_n^T)^T$ means the even partition of $\mathbf{x}$ into $n$ blocks, i.e., $x_i\in\mathbb{R}^d$ $\forall i=1,\ldots,n$. For any function $f:\mathbb{R}^d\rightarrow\mathbb{R}$, $\partial f(x)$ denotes any subgradient of $f$ at $x\in\mathbb{R}^d$, i.e., $f(y)-f(x)-\partial f(x)^T(y-x)\ge0$ $\forall y\in\mathbb{R}^d$. If $f$ is differentiable, then $\nabla f(x)$ denotes the gradient of $f$ at $x\in\mathbb{R}^d$. In addition, $I_d$ is the $d\times d$ identity matrix, $O_d$ is the $d\times d$ zero matrix, $\mathbf{1}_d\in\mathbb{R}^d$ is the all-one vector, $\mathbf{0}_d\in\mathbb{R}^d$ is the all-zero vector, and $\otimes$ is the Kronecker product. For any matrices $M,M'\in\mathbb{R}^{n\times n}$, $M\preceq M'$ and $M'\succeq M$ both mean $M'-M$ is positive semidefinite. Also, $[M]_{ij}$ represents the $(i,j)$-entry of $M$, $\mathcal{R}(M)$ the range of $M$, and $\operatorname{Null}(M)$ the null space of $M$. If $M$ is a block diagonal matrix with diagonal blocks $M_1,\ldots,M_m$, we write it as $M=\operatorname{diag}(M_1,\ldots,M_m)$. If $M$ is symmetric positive semidefinite, we use $\lambda_i^{\downarrow}(M)\ge0$ to denote its $i$th largest eigenvalue and $M^{\dag}$ its Moore-Penrose pseudoinverse. 

%
%

\section{Problem Formulation}\label{sec:probform}

Consider a set $\mathcal{V}=\{1,2,\ldots,n\}$ of agents, where each agent $i\in\mathcal{V}$ possesses a local objective function $f_i:\mathbb{R}^d\rightarrow\mathbb{R}$ and a local constraint set $X_i\subseteq\mathbb{R}^d$. All of the $n\ge2$ agents attempt to solve the constrained optimization problem
\begin{align}
\begin{array}{ll}\underset{x\in\mathbb{R}^d}{\mbox{minimize}} & \sum_{i\in\mathcal{V}}f_i(x)\\ \operatorname{subject\,to} & x\in \bigcap_{i\in\mathcal{V}}X_i,\end{array}\label{eq:problem}
\end{align}
which satisfies the following assumption.

\begin{assumption}\label{asm:problem}
(a) Each $f_i$, $i\in\mathcal{V}$ is strongly convex over $X_i$ with convexity parameter $\theta_i>0$, i.e., for any $x,y\in X_i$ and any subgradient $\partial f_i(x)$ of $f_i$ at $x$, $f_i(y)-f_i(x)-\partial f_i(x)^T(y-x)\ge\theta_i\|y-x\|^2/2$.\\
(b) $\mathbf{0}_d\in\operatorname{int}\bigcap_{i\in\mathcal{V}}X_i$.
\end{assumption}

Assumption~\ref{asm:problem} ensures the existence of a unique optimal solution $x^\star\in\bigcap_{i\in\mathcal{V}}X_i$ to problem~\eqref{eq:problem}. Notice that Assumption~\ref{asm:problem}(a) is a common assumption for distributed optimization methods with convergence rate guarantees (e.g., \cite{Giselsson13,Patrinos13,Beck14,Necoara14,Nedic16b,Nedic17}). \xw{In addition, unlike many existing works that require each $f_i$ to be continuously differentiable (e.g., \cite{QuG17b,Jakovetic14,ShiW15,QuG17,XiC17,Nedic16b,Nedic17,Koshal11,ChangTH14,Necoara14,Wei13a,Wei13b,Varagnolo16,LuJ12,LouY16}), here each $f_i$ is not necessarily differentiable.} Also, Assumption~\ref{asm:problem}(b) can always be replaced with the less restrictive condition $\operatorname{int}\bigcap_{i\in\mathcal{V}}X_i\neq\emptyset$, which is also assumed in \cite{Nedic10,Lee13,LinP16,Margellos16}. To see this, suppose $x'\in\operatorname{int}\bigcap_{i\in\mathcal{V}}X_i$ for some $x'\neq\mathbf{0}_d$. Consider the change of variable $z=x-x'$, and write each $f_i(x)$ and $X_i$ as $f_i(z+x')$ and $\{z\in\mathbb{R}^d:z+x'\in X_i\}$, respectively. Then, the resulting new problem with the decision variable $z$ is in the form of \eqref{eq:problem} and satisfies Assumption~\ref{asm:problem}.

We model the $n$ agents and their interactions as an undirected graph $\mathcal{G}^k=(\mathcal{V},\mathcal{E}^k)$ with time-varying topologies, where $k\in\{0,1,\ldots\}$ represents time, $\mathcal{V}=\{1,2,\ldots,n\}$ is the set of nodes (i.e., the agents), and $\mathcal{E}^k\subseteq\{\{i,j\}:i,j\in\mathcal{V},i\neq j\}$ is the set of links (i.e., the agent interactions) at time $k$. Without loss of generality, we assume that $\mathcal{E}^k\neq\emptyset$ $\forall k\ge0$. In addition, for each node $i\in\mathcal{V}$, let $\mathcal{N}_i^k=\{j\in\mathcal{V}:\{i,j\}\in\mathcal{E}^k\}$ be the set of its neighbors (i.e., the nodes that it directly communicates with) at time $k$.

To enable cooperation of the nodes, we need to impose an assumption on network connectivity, so that the local decisions of the nodes can be mixed across the network. To this end, define $\mathcal{E}_{\infty}:=\{\{i,j\}:\{i,j\}\in \mathcal{E}^k~\text{for infinite many}~k\geq 0\}$. Then, consider the following assumption.

\begin{assumption}[Infinite connectivity]\label{asm:infiniteconnect}
The graph $(\mathcal{V}, \mathcal{E}_{\infty})$ is connected.
\end{assumption}

Assumption~\ref{asm:infiniteconnect} is equivalent to the connectivity of the graph $(\mathcal{V},\cup_{t=k}^\infty\mathcal{E}^t)$ for all $k\ge0$. \xw{This is a minimal connectivity condition for distributed optimization algorithms to converge to optimality, which ensures every node to directly or indirectly influence any other nodes infinitely many times \cite{Nedic10}. As Assumption~\ref{asm:infiniteconnect} does not quantify how quickly the local decisions of the nodes diffuse throughout the network, we need a stronger connectivity condition to derive performance guarantees for the algorithms to be developed.}

\begin{assumption}[$B$-connectivity]\label{asm:Bconnected}
\xw{There exists an integer $B>0$ such that for any integer $k\ge0$, the graph $(\mathcal{V},\bigcup_{t=kB}^{(k+1)B-1}\mathcal{E}^t)$ is connected.}
\end{assumption}

Assumption~\ref{asm:Bconnected} forces each node to have an impact on the others in the time intervals $[kB,(k+1)B-1]$ $\forall k\ge0$ of length $B$. Compared with Assumption~\ref{asm:infiniteconnect}, Assumption~\ref{asm:Bconnected} is more restrictive but more commonly adopted in the literature (e.g., \cite{Ram09,Nedic10,ZhuM12,Lee13,ChangTH14,Nedic15,LinP16,Nedic16b,LouY16,Margellos16,Nedic17}). 


%
%
\section{Fenchel Dual Gradient Algorithms}\label{sec:dualalgo}

In this section, we develop a family of distributed algorithms to solve \eqref{eq:problem} based on Fenchel duality.

\subsection{Fenchel Dual Problem}\label{ssec:Fenchel}

We first transform \eqref{eq:problem} into the following equivalent problem:
\begin{align}
\begin{array}{ll}\underset{\mathbf{x}\in\mathbb{R}^{nd}}{\mbox{minimize}} & F(\mathbf{x}):=\sum_{i\in\mathcal{V}}f_i(x_i)\\ \operatorname{subject\,to} & x_i\in X_i,\quad\forall i\in\mathcal{V},\\ & \mathbf{x}\in S, \end{array}\label{eq:equivalentproblem}
\end{align}
where $\mathbf{x}=(x_1^T,\ldots,x_n^T)^T$ and $S:=\{\mathbf{x}\in\mathbb{R}^{nd}:x_1=x_2=\cdots=x_n\}$. Note that problem~\eqref{eq:equivalentproblem} has a unique optimal solution $\mathbf{x}^\star=((x^\star)^T,\ldots,(x^\star)^T)^T$, where $x^\star\in\bigcap_{i\in\mathcal{V}}X_i$ is the unique optimum of problem~\eqref{eq:problem}. In addition, its optimal value $F^\star$ is equal to that of problem~\eqref{eq:problem}.

Next, we construct the Fenchel dual problem \cite{Bertsekas99} of \eqref{eq:equivalentproblem}. To this end, we introduce a function $q_i:\mathbb{R}^d\times\mathbb{R}^d\rightarrow\mathbb{R}$ for each $i\in\mathcal{V}$ defined as
\begin{align*}
q_i(x_i,w_i)=w_i^Tx_i-f_i(x_i).
\end{align*}
The conjugate convex function $d_i:\mathbb{R}^d\rightarrow\mathbb{R}$ is then given by
\begin{align*}
d_i(w_i) = \sup_{x_i\in X_i}q_i(x_i,w_i).
\end{align*}
With the above, the Fenchel dual problem of \eqref{eq:equivalentproblem} can be described as
\begin{align}
\begin{array}{ll}\underset{\mathbf{w}\in\mathbb{R}^{nd}}{\mbox{maximize}} & -D(\mathbf{w}):=-\sum_{i\in\mathcal{V}} d_i(w_i)\\ \operatorname{subject\,to} & \mathbf{w}\in S^\bot, \end{array}\label{eq:dualprob}
\end{align}
where $\mathbf{w}=(w_1^T,\ldots,w_n^T)^T$ and $S^\bot:=\{\mathbf{w}\in\mathbb{R}^{nd}:w_1+w_2+\cdots+w_n=\mathbf{0}_d\}$ is the orthogonal complement of $S$. Note that \eqref{eq:dualprob} is a convex optimization problem. Also, with Assumption~\ref{asm:problem}, it can be shown that strong duality between \eqref{eq:equivalentproblem} and \eqref{eq:dualprob} holds, i.e., the optimal value $-D^\star$ of \eqref{eq:dualprob} equals $F^\star$, and that the optimal set of \eqref{eq:dualprob} is nonempty \cite{Bertsekas99}. Moreover, $\mathbf{w}^\star=((w_1^\star)^T,\ldots,(w_n^\star)^T)^T\in S^\bot$ is an optimal solution to \eqref{eq:dualprob} if and only if $\nabla d_i(w_i^\star)=\nabla d_j(w_j^\star)$ $\forall i,j\in\mathcal{V}$ \cite[Lemma 3.1]{Lakshmanan08}, i.e., $\nabla D(\mathbf{w}^\star)\in S$.

Below we acquire a couple of properties regarding the Fenchel dual problem~\eqref{eq:dualprob}. Notice from Assumption~\ref{asm:problem}(a) that for each $i\in\mathcal{V}$ and each $w_i\in\mathbb{R}^d$, there uniquely exists
\begin{align}
\tilde{x}_i(w_i):=\operatorname{arg\;max}_{x\in X_i} q_i(x,w_i).\label{eq:primalsolution}
\end{align}
Thus, $d_i$ is differentiable \cite{LuJ16} and
\begin{align}
\nabla d_i(w_i) = \tilde{x}_i(w_i).\label{eq:dual_gradient}
\end{align}
The following proposition shows that $d_i$ is smooth, i.e., $\nabla d_i$ is Lipschitz.

\begin{proposition}{\cite[Lemma II.1]{Beck14}}\label{pro:Lipschitz}
Suppose Assumption~\ref{asm:problem} holds. Then, for each $i\in\mathcal{V}$, $\nabla d_i$ is Lipschitz continuous with Lipschitz constant $L_i=1/\theta_i$, where $\theta_i>0$ is defined in Assumption~\ref{asm:problem}, i.e., $\|\nabla d_i(u_i)-\nabla d_i(v_i)\|\le L_i\|u_i-v_i\|$ $\forall u_i,v_i\in\mathbb{R}^d$.
\end{proposition}

\xw{In fact, the strong convexity of $f_i$ on $X_i$ assumed in Assumption~\ref{asm:problem}(a) is both sufficient and necessary for the smoothness of $d_i$ \cite{Beck14}.} 

Likewise, we can see that $D(\mathbf{w})$ is differentiable and
\begin{align}
\nabla D(\mathbf{w})=\tilde{\mathbf{x}}(\mathbf{w}):=(\tilde{x}_1(w_1)^T,\ldots,\tilde{x}_n(w_n)^T)^T.\label{eq:dualgradfull}
\end{align}
\xw{According to \eqref{eq:primalsolution} and \eqref{eq:dualgradfull}, if each $w_i$ is known to node $i$, then the gradient of the Fenchel dual function $D$ can be evaluated in parallel by the nodes, while the Lagrange dual of (equivalent forms of) problem~\eqref{eq:equivalentproblem} does not have such a favorable feature when the network is time-varying and not necessarily connected at each time instance.} Further, notice that $F(\mathbf{x})$ in problem~\eqref{eq:equivalentproblem} is strongly convex over $X_1\times\cdots\times X_n$ with convexity parameter $\theta_{\min}:=\min_{i\in\mathcal{V}}\theta_i$. Also note that $D(\mathbf{w})=\sup_{\mathbf{x}\in X_1\times\cdots\times X_n}\mathbf{w}^T\mathbf{x}-F(\mathbf{x})$. Like Proposition~\ref{pro:Lipschitz}, we can establish the Lipschitz continuity of $\nabla D$.

\begin{corollary}\label{cor:Lipschitz}
Suppose Assumption~\ref{asm:problem} holds. Then, $\nabla D$ is Lipschitz continuous with Lipschitz constant $L=1/\theta_{\min}$.
\end{corollary}

Finally, we show that the dual optimal set and the level sets of $D$ on $S^\bot$ are bounded.
\begin{proposition}\label{pro:boundedlevelset}
Suppose Assumption~\ref{asm:problem} holds. For any optimal solution $\mathbf{w}^\star\in S^\bot$ of problem~\eqref{eq:dualprob},
\begin{align}
\|\mathbf{w}^\star\|\le\frac{(\sum_{i\in\mathcal{V}}\max_{x_i\in B(\mathbf{0}_d,r_c)}f_i(x_i))-F^\star}{r_c}<\infty,\label{eq:dualoptimumbound}
\end{align}
where $r_c\in(0,\infty)$ is such that $B(\mathbf{0}_d,r_c)\subseteq\bigcap_{i\in\mathcal{V}}X_i$. In addition, for any $\mathbf{w}\in S^\bot$, the level set $S_0(\mathbf{w}):=\{\mathbf{w}'\in S^\bot:D(\mathbf{w}')\leq D(\mathbf{w})\}$ is compact.
\end{proposition}

\begin{proof}
See Appendix~\ref{ssec:proofofpro:boundedlevelset}.
\end{proof}

\xw{The boundedness of the dual optimal set relies on the nonemptyness of $\operatorname{int}\bigcap_{i\in\mathcal{V}}X_i$ assumed by Assumption~\ref{asm:problem}(b), without which the dual optimal set can be unbounded (e.g., $X_i=\{(z_1,z_2)^T\in\mathbb{R}^2:z_1=0\}$ $\forall i\in\mathcal{V}$).}

\subsection{Algorithms}\label{ssec:algorithm}

In \cite{XiaoL06b,Lakshmanan08}, a set of weighted gradient methods are proposed to solve a network resource allocation problem, which can be cast in the form of \eqref{eq:dualprob}. Inspired by this, we consider a class of weighted gradient methods as follows: Starting from an arbitrary $\mathbf{w}^0\in S^\bot$, the subsequent iterates are generated by
\begin{align}
\mathbf{w}^{k+1}=\mathbf{w}^k-\alpha^k(H_{\mathcal{G}^k}\otimes I_d)\nabla D(\mathbf{w}^k),\quad\forall k\ge0,\label{eq:weightgrad}
\end{align}
where $\alpha^k>0$ is the step-size and $H_{\mathcal{G}^k}\in\mathbb{R}^{n\times n}$ is the weight matrix that depends on the topology of $\mathcal{G}^k$, defined as
\begin{align}
[H_{\mathcal{G}^k}]_{ij}=\begin{cases}\sum\limits_{s\in\mathcal{N}_i^k} h_{is}^k, & \text{if }i=j,\\-h_{ij}^k,& \text{if }\{i,j\}\in\mathcal{E}^k,\\0, & \text{otherwise},\end{cases}
\quad\forall i,j\in\mathcal{V}.
\label{eq:weightmatrix}
\end{align}
We require $h_{ij}^k=h_{ji}^k>0$ $\forall\{i,j\}\in \mathcal{E}^k$ $\forall k\ge0$. We also assume that there exists a finite interval $[\underline{h},\bar{h}]$ such that
\begin{align}
h_{ij}^k\in [\underline{h},\bar{h}]\subset(0,\infty),\quad\forall k\ge0,\;\forall i\in\mathcal{V},\;\forall j\in\mathcal{N}_i^k.\label{eq:weight}
\end{align}
Since $\mathcal{E}^k\neq\emptyset$, $H_{\mathcal{G}^k}\neq O_n$ for any $k\ge0$. Moreover, $H_{\mathcal{G}^k}$ is symmetric positive semidefinite and $H_{\mathcal{G}^k}\mathbf{1}_n=\mathbf{0}_n$. Thus, using the same rationale as \cite{XiaoL06b,Lakshmanan08}, the proposition below shows that as long as $\mathbf{w}^0$ is feasible, so are $\mathbf{w}^k$ $\forall k\ge1$.

\begin{proposition}\label{pro:dualfeasible}
Let $(\mathbf{w}^k)_{k=0}^\infty$ be the iterates generated by \eqref{eq:weightgrad}. If $\mathbf{w}^0\in S^\bot$, then $(\mathbf{w}^k)_{k=0}^\infty\subseteq S^\bot$.
\end{proposition}

\begin{rem}\label{rem:resourceallocation}
The weighted gradient method \eqref{eq:weightgrad} can be tuned to solve problems of minimizing $\sum_{i\in\mathcal{V}}d_i(w_i)$ subject to $\sum_{i\in\mathcal{V}}w_i=c$, $\forall c\in\mathbb{R}^d$. To do so, we can simply replace the initial condition $\mathbf{w}^0\in S^\bot$ with $\sum_{i\in\mathcal{V}}w_i^0=c$.
\end{rem}

Next, we introduce primal iterates to the weighted gradient method \eqref{eq:weightgrad} that is intended for the Fenchel dual problem~\eqref{eq:dualprob}. Note from \eqref{eq:weightmatrix} and \eqref{eq:dualgradfull} that \eqref{eq:weightgrad} can be written as
\begin{align*}
&x_i^k=\tilde{x}_i(w_i^k),\quad\forall i\in\mathcal{V},\displaybreak[0]\\
&w_i^{k+1}=w_i^k-\alpha^k\sum_{j\in\mathcal{N}_i^k}h_{ij}^k(x_i^k-x_j^k),\quad\forall i\in\mathcal{V},
\end{align*}
where $w_i^k\in\mathbb{R}^d$ is the $i$th $d$-dimensional block of $\mathbf{w}^k$ and $\tilde{x}_i(w_i^k)$ is defined in \eqref{eq:primalsolution}. We assign each $w_i^k$ and $x_i^k$ to node $i$ as its dual and primal iterates, with $x_i^k$ being node $i$'s estimate on the optimal solution $x^\star$ of problem~\eqref{eq:problem}. Thus, the above algorithm with both dual and primal iterates can be implemented in a distributed and possibly asynchronous way on the time-varying network, as is shown in Algorithm~\ref{alg:weightgrad}.

{
\renewcommand{\baselinestretch}{1.05}
\begin{algorithm} [ht]
\caption{\small Fenchel Dual Gradient Method}
\label{alg:weightgrad}
\begin{algorithmic}[1]
\small
\STATE \textbf{Initialization:} Each node $i\in\mathcal{V}$ selects $w_i^0\in\mathbb{R}^d$ so that $\sum_{j\in\mathcal{V}}w_j^0=\mathbf{0}_d$ (or simply sets $w_i^0=\mathbf{0}_d$), and sets
$x_i^{0}=\operatorname{arg\;max}_{x\in X_i}(w_i^{0})^Tx-f_i(x)$.
\FOR{ $k=0,1,\ldots$}
\STATE Each node $i\in\mathcal{V}$ with $\mathcal{N}_i^k\neq\emptyset$ sends its $x_i^k$ to all $j\in\mathcal{N}_i^k$.
\STATE Upon receiving $x_j^k$ $\forall j\in\mathcal{N}_i^k$, each node $i\in\mathcal{V}$ with $\mathcal{N}_i^k\neq\emptyset$ updates
$w_i^{k+1}=w_i^k-\alpha^k\sum_{j\in\mathcal{N}_i^k}h_{ij}^k(x_i^k-x_j^k)$.
\STATE Each node $i\in\mathcal{V}$ with $\mathcal{N}_i^k\neq\emptyset$ computes
$x_i^{k+1}=\operatorname{arg\;max}_{x\in X_i}(w_i^{k+1})^Tx-f_i(x)$.
\STATE Each node $i\in\mathcal{V}$ with $\mathcal{N}_i^k=\emptyset$ takes no action, i.e., $w_i^{k+1}=w_i^k$ and $x_i^{k+1}=x_i^{k}$.
\ENDFOR
\end{algorithmic}
\end{algorithm}
}
In Algorithm~\ref{alg:weightgrad}, the initial condition $\mathbf{w}^0\in S^\bot$ can simply be realized by setting $w_i^0=\mathbf{0}_d$ $\forall i\in\mathcal{V}$. Subsequently at each iteration, every node $i$ with at least one neighbor updates its dual iterate $w_i^k$ via local interactions with its current neighbors and then updates its primal iterate $x_i^k$ on its own. 

To implement Algorithm~\ref{alg:weightgrad}, each node $i$ needs to select the weights $h_{ij}^k$ $\forall j\in\mathcal{N}_i^k$ that satisfy $h_{ij}^k=h_{ji}^k$ in a predetermined interval $[\underline{h},\bar{h}]\subset(0,\infty)$, where $\underline{h}$ and $\bar{h}$ may or may not be related with $\mathcal{G}^k$ $\forall k\ge0$. This can be done through inexpensive interactions between neighboring nodes. Two typical examples of $H_{\mathcal{G}^k}$ are the graph Laplacian matrix
\begin{align}
[H_{\mathcal{G}^k}]_{ij}=[L_{\mathcal{G}^k}]_{ij}:=\begin{cases}|\mathcal{N}_i^k|, & \text{if }i=j,\\-1,& \text{if }\{i,j\}\in\mathcal{E}^k,\\0, & \text{otherwise,}\end{cases}\label{eq:Laplacian}
\end{align}
and the Metropolis weight matrix \cite{XiaoL06b}
\begin{align}
[H_{\mathcal{G}^k}]_{ij}=\begin{cases}\sum\limits_{s\in\mathcal{N}_i^k} \frac{1}{\max\{|\mathcal{N}_i^k|L_i, |\mathcal{N}_s^k|L_s\}}, & \text{if }i=j,\\-\frac{1}{\max\{|\mathcal{N}_i^k|L_i, |\mathcal{N}_j^k|L_j\}},& \text{if }\{i,j\}\in\mathcal{E}^k,\\0, & \text{otherwise}.\end{cases}
\label{eq:localweightmatrix}
\end{align}
When $H_{\mathcal{G}^k}$ is set to \eqref{eq:Laplacian}, each node $i$ does not need any additional efforts in computing the weights $h_{ij}^k$ $\forall j\in\mathcal{N}_i^k$ since they are $1$ by default. When $H_{\mathcal{G}^k}$ is set to \eqref{eq:localweightmatrix}, each node $i$ only needs to obtain from every neighbor $j\in\mathcal{N}_i^k$ the product of node $j$'s neighborhood size $|\mathcal{N}_j^k|$ and Lipschitz constant $L_j=1/\theta_j$ of $\nabla d_j$. 

The remaining parameter to be determined is the step-size $\alpha^k$. Later in Section~\ref{sec:convanal}, we will show that the following step-size condition is sufficient to guarantee the convergence of Algorithm~\ref{alg:weightgrad}: Suppose there is a finite interval $[\underline{\alpha},\bar{\alpha}]$ such that
\begin{align}
\alpha^k\in[\underline{\alpha},\bar{\alpha}]\subset(0,2/\delta),\quad\forall k\ge0,\label{eq:stepsize}
\end{align}
where $\delta>0$ can be any positive constant satisfying
\begin{align}
H_{\mathcal{G}^k}\preceq\delta\Lambda_L^{-1},\quad\forall k\geq 0,\label{eq:delta}
\end{align}
with $\Lambda_L:=\operatorname{diag}(L_1,\ldots,L_n)$.
Note that such $\delta$ always exists because $\Lambda_L^{-1}$ is positive definite and $H_{\mathcal{G}^k}$ is positive semidefinite. For example, we may choose $\delta=L\sup_{k\geq 0} \lambda_1^{\downarrow}(H_{\mathcal{G}^k})$, where $L=1/\theta_{\min}=\max_{i\in\mathcal{V}}L_i$. More conservatively, because $H_{\mathcal{G}^k}\preceq\bar{h}L_{\mathcal{G}^k}$ and $\lambda_1^{\downarrow}(L_{\mathcal{G}^k})\le n$, we can always let $\delta=L\bar{h}n$ and thus
\begin{align*}
[\underline{\alpha},\bar{\alpha}]\subset(0,\frac{2}{L\bar{h}n}).
\end{align*}
Since $\bar{h}$ can be predetermined and known to all the nodes, this condition only requires the nodes to obtain the global quantities $n$ and $L=\max_{i\in\mathcal{V}}L_i$, which can be computed decentralizedly by some consensus schemes (e.g., \cite{ChenJY06}). Below, we provide less conservative step-size conditions for the two specific choices of $H_{\mathcal{G}^k}$ in \eqref{eq:Laplacian} and \eqref{eq:localweightmatrix}, which also can be satisfied by the nodes without any centralized coordination.

\begin{example}\label{ex:stepsizeLaplacian}
When $H_{\mathcal{G}^k}$ is set to the graph Laplacian matrix $L_{\mathcal{G}^k}$ as in \eqref{eq:Laplacian}, in addition to the aforementioned choice $\delta=L\sup_{k\geq 0} \lambda_1^{\downarrow}(L_{\mathcal{G}^k})$, another option for $\delta$ could be $\delta\!=\!2\displaystyle{\sup_{k\ge0}\max_{i\in\mathcal{V}}}|\mathcal{N}_i^k|L_i$, so that $\delta\Lambda_L^{-1}-L_{\mathcal{G}^k}$ is diagonally dominant and thus positive semidefinite for each $k\geq 0$. Therefore, $\alpha^k$ can be selected in the interval $[\underline{\alpha},\bar{\alpha}]$ satisfying
\begin{align*}
0<\underline{\alpha}\le\bar{\alpha}<&\frac{1}{\min\{\frac{L}{2}\sup\limits_{k\geq 0} \lambda_1^{\downarrow}(L_{\mathcal{G}^k}),\sup\limits_{k\ge0}\max\limits_{i\in\mathcal{V}}|\mathcal{N}_i^k|L_i\}}.
\end{align*}
The above step-size condition can be simplified for some special interaction patterns. For instance, if the nodes interact in a gossiping pattern, i.e., each $\mathcal{E}^k$ contains only one link, then we may let $0<\underline{\alpha}\le\bar{\alpha}<1/L$.
Even though the topologies of $(\mathcal{G}^k)_{k=0}^\infty$ are completely unknown, since $\lambda_1^{\downarrow}(L_{\mathcal{G}^k})\le n$, we can adopt a more conservative step-size condition $0<\underline{\alpha}\le\bar{\alpha}<2/(nL)$. 
\end{example}

\begin{example}\label{ex:stepsizeMetropolis}
When $H_{\mathcal{G}^k}$ is set according to \eqref{eq:localweightmatrix}, we can simply take $\delta=2$, because $2\Lambda_L^{-1}-H_{\mathcal{G}^k}$ is diagonally dominant and thus $2\Lambda_L^{-1}\succeq H_{\mathcal{G}^k}$. Hence, the step-sizes can be selected as
\begin{align*}
0<\underline{\alpha}\le\alpha^k\le\bar{\alpha}<1,\quad\forall k\ge0,
\end{align*}
which requires no global information and is independent of the network and the problem.
\end{example}

The underlying weighted gradient method \eqref{eq:weightgrad} in Algorithm~\ref{alg:weightgrad} can be viewed as a generalization of the distributed weighted gradient methods in \cite{XiaoL06b,Lakshmanan08}. By assuming the (directed) network to be time-invariant and connected, \cite{XiaoL06b} proposes a class of weighted gradient methods in the form of \eqref{eq:weightgrad} but with a constant weight matrix. It is also shown in \cite{XiaoL06b} that if the time-invariant network is further undirected, the constant weight matrix can be determined in a distributed fashion via \eqref{eq:Laplacian} or \eqref{eq:localweightmatrix}. The step-size conditions in \cite{XiaoL06b} for fixed undirected networks and fixed weight matrices given by \eqref{eq:Laplacian} and \eqref{eq:localweightmatrix} are extended here in Examples~\ref{ex:stepsizeLaplacian} and~\ref{ex:stepsizeMetropolis} to handle time-varying networks and time-varying weight matrices. On the other hand, \cite{Lakshmanan08} considers time-varying undirected networks satisfying Assumption~\ref{asm:Bconnected}. By setting $H_{\mathcal{G}^k}$ to $L_{\mathcal{G}^k}$ in \eqref{eq:Laplacian} and $\alpha^k=1/(2nL)$ $\forall k\ge0$, \eqref{eq:weightgrad} reduces to the algorithm in \cite{Lakshmanan08}. Note from Example~\ref{ex:stepsizeLaplacian} that here we allow for a much broader step-size range for this particular weight matrix.
%
%
\section{Convergence Analysis}\label{sec:convanal}

This section is dedicated to analyzing the convergence performance of Algorithm~\ref{alg:weightgrad}.

%
%
\subsection{Asymptotic convergence under infinite connectivity}\label{ssec:InfiniteGraphs}

In this subsection, we show that Algorithm~\ref{alg:weightgrad} asymptotically converges to the optimum of problem~\eqref{eq:problem} under Assumption~\ref{asm:infiniteconnect}. 

We first show that \xw{the step-size condition~\eqref{eq:stepsize} ensures $(D(\mathbf{w}^k))_{k=0}^\infty$ to be non-increasing.}

\begin{lemma}\label{lemma:funcvaldescent}
Suppose Assumption \ref{asm:problem} holds. Let $(\mathbf{w}^k)_{k=0}^{\infty}$ be the dual iterates generated by Algorithm~\ref{alg:weightgrad}. If the step-sizes $(\alpha^k)_{k=0}^{\infty}$ satisfy \eqref{eq:stepsize}, then for each $k\ge0$,
\begin{align*}
D(\mathbf{w}^{k+1})-D(\mathbf{w}^k)\leq-\rho\nabla D(\mathbf{w}^k)^T(H_{\mathcal{G}^k}\otimes I_d)\nabla D(\mathbf{w}^k),
\end{align*}
where $\rho:=\min\{\underline{\alpha}-\frac{\underline{\alpha}^2\delta}{2}, \bar{\alpha}-\frac{\bar{\alpha}^2\delta}{2}\}\in(0,\infty)$, with $\underline{\alpha},\bar{\alpha}>0$ in \eqref{eq:stepsize} and $\delta>0$ in \eqref{eq:delta}.
\end{lemma}

\begin{proof}
See Appendix~\ref{ssec:proofoflemma:funcvaldescent}.
\end{proof}

Lemma~\ref{lemma:funcvaldescent}, along with Propositions~\ref{pro:boundedlevelset} and~\ref{pro:dualfeasible}, implies that for each $k\ge0$, $\mathbf{w}^k\in S_0(\mathbf{w}^0)$ and $\|\mathbf{w}^k-\mathbf{w}^\star\|\le M_0$, where $\mathbf{w}^\star$ is any optimum of problem~\eqref{eq:dualprob} and
\begin{align}
M_0:=\!\!\!\!\!\!\max_{\mathbf{w}\in S_0(\mathbf{w}^0),\;\mathbf{w}^\star\in S^\bot:D(\mathbf{w}^\star)=D^\star}\|\mathbf{w}-\mathbf{w}^\star\|\in[0,\infty).\label{eq:M0}
\end{align}
Another important consequence of Lemma~\ref{lemma:funcvaldescent} is that the differences of the primal iterates along the time-varying links are vanishing. To see this, by adding the inequality in Lemma~\ref{lemma:funcvaldescent} from $k=0$ to $\infty$,
\begin{align*}
\sum_{k=0}^{\infty}\!\langle\mathbf{x}^k,(H_{\mathcal{G}^k}\!\otimes\!I_d)\mathbf{x}^k\rangle&\!=\!\sum_{k=0}^{\infty}\!\langle\nabla D(\mathbf{w}^k),(H_{\mathcal{G}^k}\!\otimes\!I_d)\nabla D(\mathbf{w}^k)\rangle\displaybreak[0]\\
&\leq(D(\mathbf{w}^0)-D^{\star})/\rho<\infty,
\end{align*}
where $\mathbf{x}^k=((x_1^k)^T,\ldots,(x_n^k)^T)^T$. This implies that $\langle\mathbf{x}^k,(H_{\mathcal{G}^k}\otimes I_d)\mathbf{x}^k\rangle\rightarrow 0$ as $k\rightarrow \infty$. Since
$\langle\mathbf{x}^k, (H_{\mathcal{G}^k}\otimes I_d)\mathbf{x}^k\rangle= \sum_{\{i,j\}\in \mathcal{E}^k} h_{ij}^k\|x_i^k-x_j^k\|^2$ and $h_{ij}^k\ge\underline{h}>0$ $\forall\{i,j\}\in\mathcal{E}^k$, we have
\begin{align}
\lim_{k\rightarrow\infty}\max_{\{i,j\}\in\mathcal{E}^k}\|x_i^k-x_j^k\|=0.\label{eq:limmaxxx=0}
\end{align}
Because $\mathcal{G}^k$ may not be connected at each $k\ge0$, \eqref{eq:limmaxxx=0} alone is insufficient to assert that the primal iterates $x_i^k$ $\forall i\in\mathcal{V}$ asymptotically reach a consensus. Nevertheless, by integrating \eqref{eq:limmaxxx=0} with Assumption~\ref{asm:infiniteconnect}, we are able to show in Lemma~\ref{lemma:convergesgradient} below that such an assertion is indeed true. The main idea of proving this can be summarized as follows: By \eqref{eq:limmaxxx=0} we know that $\|x_i^k-x_j^k\|$ $\forall\{i,j\}\in \mathcal{E}^k$ can be arbitrarily small after some time $T\ge0$. Then, instead of studying the differences $\|x_i^k-x_j^k\|$ $\forall i,j\in\mathcal{V}$ across the entire network, we show that such differences within each connected component of the graph $(\mathcal{V},\cup_{t=T}^k\mathcal{E}^t)$ become sufficiently small after some $k\ge T$. Finally, note from Assumption~\ref{asm:infiniteconnect} that the graph $(\mathcal{V},\cup_{t=T}^k\mathcal{E}^t)$ must be connected when $k\ge T$ is sufficiently large. The dissipation of the differences among all the $x_i^k$'s can thus be concluded.

\begin{lemma}\label{lemma:convergesgradient}
Suppose Assumptions~\ref{asm:problem} and \ref{asm:infiniteconnect} hold. Let $(\mathbf{x}^k)_{k=0}^{\infty}$ be the primal iterates generated by Algorithm~\ref{alg:weightgrad}. If the step-sizes $(\alpha^k)_{k=0}^{\infty}$ satisfy \eqref{eq:stepsize}, then $\lim\limits_{k\rightarrow\infty}\max\limits_{i,j\in\mathcal{V}}\|x_i^k-x_j^k\|=0$.
\end{lemma}

\begin{proof}
See Appendix~\ref{ssec:proofoflemmaconvergesgradient}.
\end{proof}

Since $x_i^k\in X_i$ $\forall i\in\mathcal{V}$, $\mathbf{x}^k$ is feasible if and only if $\mathbf{x}^k\in S$. Thus, $\|P_{S^{\bot}}(\mathbf{x}^k)\|$ can be used to quantify the infeasibility of $\mathbf{x}^k$. Note that $\|P_{S^{\bot}}(\mathbf{x}^k)\|^2=\|\mathbf{x}^k-P_S(\mathbf{x}^k)\|^2=\sum_{i\in\mathcal{V}}\|x_i^k-\frac{1}{n}\sum_{j\in\mathcal{V}}x_j^k\|^2\le\frac{1}{n}\sum_{i\in\mathcal{V}}\sum_{j\in\mathcal{V}}\|x_i^k-x_j^k\|^2$. It follows from Lemma~\ref{lemma:convergesgradient} that $\|P_{S^{\bot}}(\mathbf{x}^k)\|^2\rightarrow 0$ as $k\rightarrow\infty$. \xw{This can further be utilized to establish the asymptotic convergence to both dual and primal optimality}, as is shown in the theorem below.

\begin{theorem}\label{thm:asymconv}
Suppose Assumptions~\ref{asm:problem} and~\ref{asm:infiniteconnect} hold. Let $(\mathbf{w}^k)_{k=0}^{\infty}$ and $(\mathbf{x}^k)_{k=0}^{\infty}$ be the dual and primal iterates generated by Algorithm~\ref{alg:weightgrad}, respectively. If the step-sizes $(\alpha^k)_{k=0}^{\infty}$ satisfy \eqref{eq:stepsize}, then $\lim_{k\rightarrow\infty}\|P_{S^{\bot}}(\mathbf{x}^k)\|=0$, $\lim_{k\rightarrow\infty}D(\mathbf{w}^k)=D^{\star}$, $\lim_{k\rightarrow\infty}F(\mathbf{x}^k)=F^\star$, and $\lim_{k\rightarrow\infty}\mathbf{x}^k= \mathbf{x}^{\star}$.
\end{theorem}

\begin{proof}
See Appendix~\ref{ssec:proofofthm:asymconv}.
\end{proof}

%
%
\subsection{Convergence rates under $B$-connectivity}\label{ssec:dualBconnected}

In this subsection, we offer sublinear rates of convergence for Algorithm~\ref{alg:weightgrad} under Assumption~\ref{asm:Bconnected}.

\xw{Inspired from \cite{Lakshmanan08}, we first provide a bound on the accumulative drop in the value of $D$ over each time interval $[tB,(t+1)B-1]$, $t\in\{0,1,\ldots\}$, which depends only on the dual iterate at time $tB$ and the underlying interaction graph during these $B$ iterations. To this end, for each $k\ge0$, let $\tilde{\mathcal{G}}^{k}=(\mathcal{V},\tilde{\mathcal{E}}^{k})$ be \emph{any} spanning subgraph of $(\mathcal{V}, \bigcup_{t=k}^{k+B-1} \mathcal{E}^t)$, which, owing to Assumption~\ref{asm:Bconnected}, is chosen to be connected at $k\in\{0,B,2B,\ldots\}$. Also let $\varpi^k$ be the maximum degree of $\tilde{\mathcal{G}}^{k}$ and $\bar{\varpi}:=\sup_{t\in\{0,1,\ldots\}}\varpi^{tB}$. Clearly, $1\le\varpi^{tB}\le\bar{\varpi}\le n-1$ $\forall t\in\{0,1,\ldots\}$.} 

\begin{lemma}\label{lemma:finallemma}
Suppose Assumptions~\ref{asm:problem} and~\ref{asm:Bconnected} hold. Let $(\mathbf{w}^k)_{k=0}^{\infty}$ be the dual iterates generated by Algorithm~\ref{alg:weightgrad}. If the step-sizes $(\alpha^k)_{k=0}^{\infty}$ satisfy \eqref{eq:stepsize}, then for each $k\in\{0,B,2B,\ldots\}$,\xw{
\begin{align}
\sum_{t=k}^{k+B-1}\nabla D(\mathbf{w}^t)^T(H_{\mathcal{G}^t}\otimes I_d)\nabla D(\mathbf{w}^t)\ge\nabla D(\mathbf{w}^k)^T(L_{\tilde{\mathcal{G}}^k}\otimes I_d)\nabla D(\mathbf{w}^k)/\eta,\label{eq:finallemmaeq}
\end{align}
where $\eta:=3B\bar{\varpi}\bar{\alpha}^2\delta L+3/\underline{h}\in(0,\infty)$}, with $\bar{\alpha}>0$ in \eqref{eq:stepsize}, $\delta>0$ in \eqref{eq:delta}, $L>0$ in Corollary~\ref{cor:Lipschitz}, and $\underline{h}>0$ in \eqref{eq:weight}.
\end{lemma}

\begin{proof}
See Appendix~\ref{ssec:proofoflemma:finallemma}.
\end{proof}

\xw{When $H_{\mathcal{G}^k}=L_{\mathcal{G}^k}$ and $\alpha^k=1/(2nL)$, \cite[Lemma A.9]{Lakshmanan08} provides a similar bound to \eqref{eq:finallemmaeq} with $\eta$ replaced by $3B/2$ and $\tilde{\mathcal{G}}^k$ being a spanning tree. Lemma~\ref{lemma:finallemma} improves this bound since $\eta\le 3B/4+3$ for such a particular choice of $H_{\mathcal{G}^k}$ and $\alpha^k$, allows for more general selections of $H_{\mathcal{G}^k}$ and $\alpha^k$, and sheds light on how the network topologies come into play.}

\xw{Lemma~\ref{lemma:funcvaldescent} and Lemma~\ref{lemma:finallemma} together bound the decrease in the value of $D$ during every $B$ iterations, with which we are able to provide a rate for $D(\mathbf{w}^k)\rightarrow D^\star$.} Prior to doing that, we define a sequence $(\tilde{M}_k)_{k=0}^\infty$ as follows: Let $\tilde{M}_0\in\mathbb{R}$ be any positive constant and define
\begin{align}
\tilde{M}_k=\max_{t=0,\ldots,k-1}\min_{\mathbf{w}^\star\in S^\bot:D(\mathbf{w}^\star)=D^\star}\|\mathbf{w}^{tB}-\mathbf{w}^\star\|,\;\forall k\ge1.\label{eq:tildeMk}
\end{align}
Notice that $0\le\tilde{M}_k\le M_0<\infty$, where $M_0$ is given by \eqref{eq:M0}. 

\begin{theorem}\label{theorem:dualconvrate}
Suppose Assumptions~\ref{asm:problem} and~\ref{asm:Bconnected} hold. Let $(\mathbf{w}^k)_{k=0}^{\infty}$ be the dual iterates generated by Algorithm~\ref{alg:weightgrad}. If the step-sizes $(\alpha^k)_{k=0}^\infty$ satisfy \eqref{eq:stepsize}, then for each $k\ge0$,\xw{
\begin{align}
D(\mathbf{w}^k)-D^{\star} &\leq \frac{\eta \tilde{M}_{\lfloor k/B\rfloor}^2(D(\mathbf{w}^0)-D^\star)}{\eta \tilde{M}_{\lfloor k/B\rfloor}^2+\rho\underline{\lambda}(D(\mathbf{w}^0)-D^\star)\lfloor k/B\rfloor},\label{eq:theorem1result}
\end{align}}
where $\tilde{M}_{\lfloor k/B\rfloor}\in [0,M_0]$ is defined in \eqref{eq:tildeMk} with $M_0\ge0$ in \eqref{eq:M0}, $\underline{\lambda}:=\inf_{t\in\{0,1,\ldots\}}\lambda_{n-1}^{\downarrow}(L_{\tilde{\mathcal{G}}^{tB}})\in(0,\infty)$, and $\eta,\rho>0$ are given in Lemma~\ref{lemma:finallemma} and Lemma~\ref{lemma:funcvaldescent}, respectively.
\end{theorem}

\begin{proof}
See Appendix~\ref{ssec:proofoftheorem:dualconvrate}.
\end{proof}

Theorem~\ref{theorem:dualconvrate} says that Algorithm~\ref{alg:weightgrad}, or equivalently, the underlying weighted gradient method~\eqref{eq:weightgrad}, converges to the optimal value $D^\star$ of problem~\eqref{eq:dualprob} at an $O(1/k)$ rate. \xw{The derivation of this result requires each $d_i$ to be smooth and the dual optimal set to be compact. These two conditions on problem~\eqref{eq:dualprob} may not hold if Assumption~\ref{asm:problem} is not satisfied (cf. Section~\ref{ssec:Fenchel}). Note that without the compactness of the dual optimal set, \eqref{eq:theorem1result} still holds, but we cannot guarantee $(\mathbf{w}^k)_{k=0}^\infty$ and thus $\tilde{M}_{\lfloor k/B\rfloor}$ $\forall k\ge0$ to be bounded.

The distributed weighted gradient methods in \cite{XiaoL06b,Lakshmanan08} also require the above two conditions on problem~\eqref{eq:dualprob} to establish their convergence to $D^\star$.} By imposing an additional assumption that the Hessian matrices of $d_i$ $\forall i\in\mathcal{V}$ are positive definite, the methods in \cite{XiaoL06b} are proved to achieve linear convergence rates on fixed networks. In contrast, Theorems~\ref{thm:asymconv} and~\ref{theorem:dualconvrate} allow for time-varying networks and do not even require the existence of the Hessian matrices of $d_i$ $\forall i\in\mathcal{V}$. The algorithm in \cite{Lakshmanan08} is shown to asymptotically drive $D(\mathbf{w}^k)$ to $D^\star$ and satisfy $\min_{t=1,\ldots,k}\|P_{S^\bot}(\nabla D(\mathbf{w}^{tB}))\|^2\le C\cdot n^3B/k$ for some $C>0$. Our results in Theorems~\ref{thm:asymconv} and~\ref{theorem:dualconvrate} for the more general algorithm~\eqref{eq:weightgrad} are still stronger. We show that $\lim_{k\rightarrow\infty}D(\mathbf{w}^k)=D^\star$ under the less restrictive Assumption~\ref{asm:infiniteconnect}, and that $D(\mathbf{w}^k)$ converges to $D^\star$ at an $O(1/k)$ rate under Assumption~\ref{asm:Bconnected}. Also, since $\nabla D(\mathbf{w}^{k})=\mathbf{x}^{k}$, the first inequality in Theorem~\ref{thm:primalconvrate} below is comparable to and slightly stronger than the aforementioned convergence rate in \cite{Lakshmanan08}.

Based on Theorem~\ref{theorem:dualconvrate}, below we show that the primal errors $\|\mathbf{x}^k-\mathbf{x}^{\star}\|$ and $|F(\mathbf{x}^k)-F^{\star}|$ in optimality and $\|P_{S^\bot}(\mathbf{x}^k)\|$ in feasibility all converge to zero at rates of $O(1/\sqrt{k})$. \xw{Like many Lagrange dual gradient methods (e.g., \cite{Giselsson13,LuJ16}), we do so by relating such primal errors with the dual error $D(\mathbf{w}^k)-D^{\star}$.}

\begin{theorem}\label{thm:primalconvrate}
Suppose Assumptions~\ref{asm:problem} and~\ref{asm:Bconnected} hold. Let $(\mathbf{x}^k)_{k=0}^{\infty}$ be the primal iterates generated by Algorithm~\ref{alg:weightgrad}. If the step-sizes $(\alpha^k)_{k=0}^\infty$ satisfy \eqref{eq:stepsize}, then for each $k\ge0$,\xw{
\begin{align*}
&\|P_{S^\bot}(\mathbf{x}^k)\|\le\|\mathbf{x}^k-\mathbf{x}^{\star}\|\leq\sqrt{\frac{2L\eta \tilde{M}_{\lfloor k/B\rfloor}^2(D(\mathbf{w}^0)-D^{\star})}{\eta \tilde{M}_{\lfloor k/B\rfloor}^2\!\!+\!\!\rho\underline{\lambda}(D(\mathbf{w}^0)\!-\!D^{\star})\lfloor k/B\rfloor}},\displaybreak[0]\\
&F(\mathbf{x}^k)\!-\!F^{\star}\!\leq\! \|\mathbf{w}^k\|\!\sqrt{\frac{2L\eta \tilde{M}_{\lfloor k/B\rfloor}^2(D(\mathbf{w}^0)\!-\!D^{\star})}{\eta \tilde{M}_{\lfloor k/B\rfloor}^2\!+\!\rho\underline{\lambda}(D(\mathbf{w}^0)\!-\!D^{\star})\lfloor k/B\rfloor}},\displaybreak[0]\\
&F(\mathbf{x}^k)\!-\!F^{\star}\!\ge\!-\|\mathbf{w}^\star\|\!\sqrt{\frac{2L\eta \tilde{M}_{\lfloor k/B\rfloor}^2(D(\mathbf{w}^0)\!-\!D^{\star})}{\eta \tilde{M}_{\lfloor k/B\rfloor}^2\!+\!\rho\underline{\lambda}(D(\mathbf{w}^0)\!-\!D^{\star})\lfloor k/B\rfloor}},
\end{align*}}
where $\mathbf{w}^\star$ is any optimal solution of problem~\eqref{eq:dualprob}, $L$ is given in Corollary~\ref{cor:Lipschitz}, and the remaining constants have been introduced in Theorem \ref{theorem:dualconvrate}.
\end{theorem}

\begin{proof}
See Appendix~\ref{ssec:proofofthm:primalconvrate}.
\end{proof}

Since $\mathbf{w}^k\in S_0(\mathbf{w}^0)$ $\forall k\ge0$ and $S_0(\mathbf{w}^0)$ is compact, the term $\|\mathbf{w}^k\|$ that appears in the convergence rate of $F(\mathbf{x}^k)-F^{\star}$ is uniformly bounded above by $M_0+\|\mathbf{w}^\star\|$. Consequently, the primal convergence rates of Algorithm~\ref{alg:weightgrad} in Theorem~\ref{thm:primalconvrate} are all of order $O(1/\sqrt{k})$, which commensurate with the convergence rate of the classic (centralized) subgradient projection method \cite{Nesterov04}. 

In the final part of this section, we compare the primal convergence rates of Algorithm~\ref{alg:weightgrad} with those of the existing distributed optimization algorithms that also have guaranteed \emph{convergence rates} over \emph{time-varying} networks, including Subgradient-Push \cite{Nedic15}, Gradient-Push \cite{Nedic16b}, DIGing \cite{Nedic17}, and Push-DIGing \cite{Nedic17}. Different from Algorithm~\ref{alg:weightgrad} that is developed by applying distributed weighted gradient methods to the Fenchel dual, Subgradient-Push and Gradient-Push are constructed by incorporating the subgradient method and the stochastic gradient descent method into the Push-Sum consensus protocol \cite{Kempe03}, DIGing is designed by combining a distributed inexact gradient method with a gradient tracking technique, and Push-DIGing is derived by introducing Push-Sum into DIGing.

The convergence rates of the aforementioned algorithms are all established under Assumption~\ref{asm:Bconnected}.\footnote{\xw{When it comes to Subgradient-Push, Gradient-Push, and Push-DIGing, ``connected'' in Assumption~\ref{asm:Bconnected} is indeed ``strongly connected'' since they consider directed networks.}} For each of these algorithms, Table~\ref{table:comparison} lists its assumptions and convergence rate. Observe that only Algorithm~\ref{alg:weightgrad} is capable of solving problems with different local constraints of the agents, while the remaining algorithms all require the problem to be unconstrained and their extensions to constrained problems are still open challenges. Also, Gradient-Push, DIGing, and Push-DIGing require both strong convexity and smoothness of the $f_i$'s, leading to faster convergence rates than the $O(1/\sqrt{k})$ rate of Algorithm~\ref{alg:weightgrad}. This is natural because we assume a weaker condition on $f_i$ $\forall i\in\mathcal{V}$, which allows the strongly convex $f_i$'s to be nonsmooth. Subgradient-Push needs neither strong convexity nor smoothness of each $f_i$, and the resulting convergence rate $O(\ln k/\sqrt{k})$ is slower than our $O(1/\sqrt{k})$ result. Note that the assumption on the $f_i$'s for Algorithm~\ref{alg:weightgrad} is not necessarily more restrictive than that for Subgradient-Push, since Subgradient-Push requires the subgradients of each $f_i$ to be uniformly bounded over $\mathbb{R}^d$ but Algorithm~\ref{alg:weightgrad} does not. Unlike Subgradient-Push, Gradient-Push, and Push-DIGing that admit directed links, DIGing and Algorithm~\ref{alg:weightgrad} are only applicable to undirected graphs. With that said, Algorithm~\ref{alg:weightgrad} is guaranteed to converge to the optimum with the minimal connectivity condition, i.e., Assumption~\ref{asm:infiniteconnect}, while the other methods have no such convergence results. 

\begin{table*}[tb]
  \centering
    \begin{tabular}{|c|c|c|c|c|c|c|}
      \hline
			Algorithm & unconstrained & strongly  & Lipschitz & bounded & undirected & convergence\\
	              & problem & convex  & gradient  & subgradient & links & rate\\
      \hline
      Subgradient-Push \cite{Nedic15} & $\surd$ &  &  & $\surd$ &  & $O(\ln k/\sqrt{k})$\\
      \hline
      Gradient-Push \cite{Nedic16b} & $\surd$ & $\surd$ & $\surd$ & &  & $O(\ln k/k)$\\
      \hline
      DIGing \cite{Nedic17} & $\surd$ & $\surd$ & $\surd$ &  & $\surd$ & $O(q^k)$, $0<q<1$\\
      \hline
      Push-DIGing \cite{Nedic17} & $\surd$ & $\surd$ & $\surd$ & & & $O(q^k)$, $0<q<1$\\
      \hline
      Algorithm~\ref{alg:weightgrad} & & $\surd$ & &  & $\surd$ & $O(1/\sqrt{k})$\\
      \hline
    \end{tabular}
		\vspace*{0.1in}
  \caption{{\upshape \xw{Comparison of Algorithm~\ref{alg:weightgrad} and related methods in assumptions and convergence rate. Here, $\surd$ means the assumption is required.}}} \label{table:comparison}
\end{table*}

\section{Numerical Examples}\label{sec:numericalexample}

In this section, we demonstrate the competent convergence performance of the proposed distributed Fenchel dual gradient methods by comparing them with a number of existing distributed optimization algorithms via simulations.

\subsection{Constrained case}\label{ssec:constrained}

\begin{figure*}[tb]
\centering
\subfigure[$n=50$, $B=10$, $2<\theta_i<3$]{
\includegraphics[width=0.3\linewidth, height=0.2\linewidth]{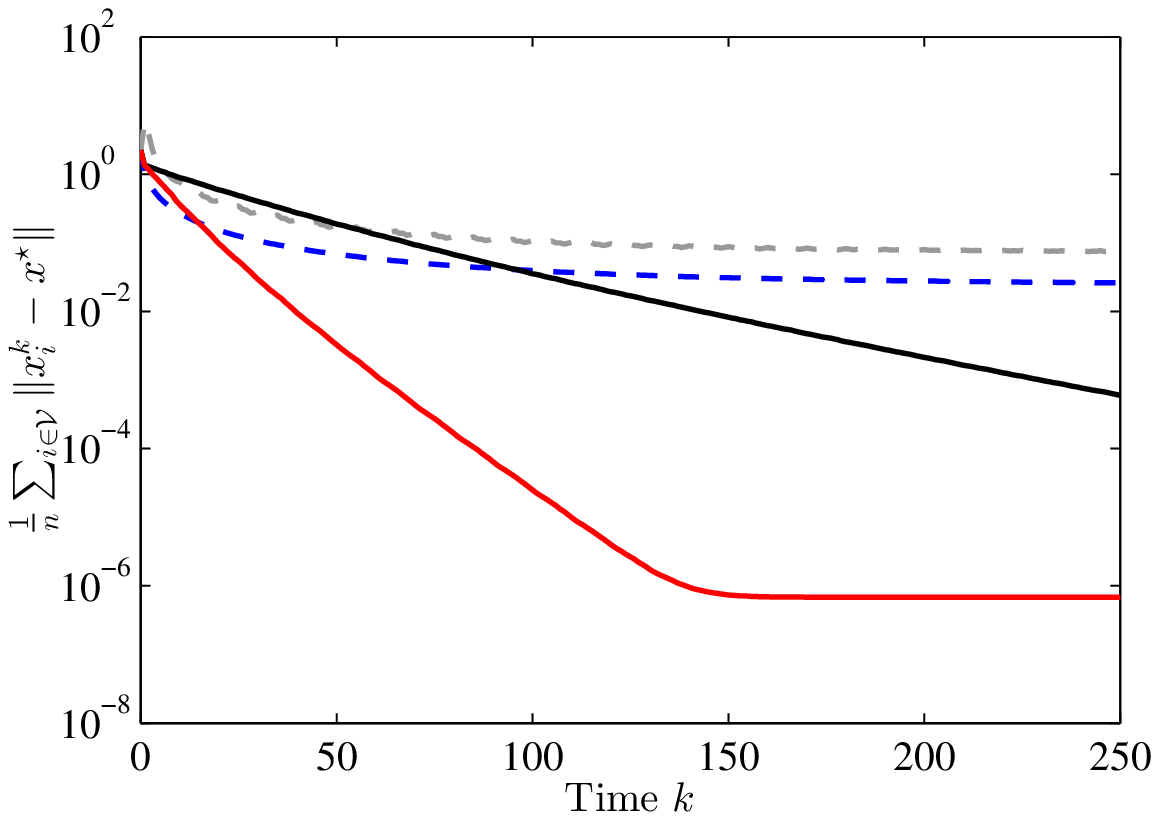}\label{fig:n50B10theta2}}
\hspace{0.01\linewidth}
\subfigure[$n=500$, $B=10$, $2<\theta_i<3$]{
\includegraphics[width=0.3\linewidth, height=0.2\linewidth]{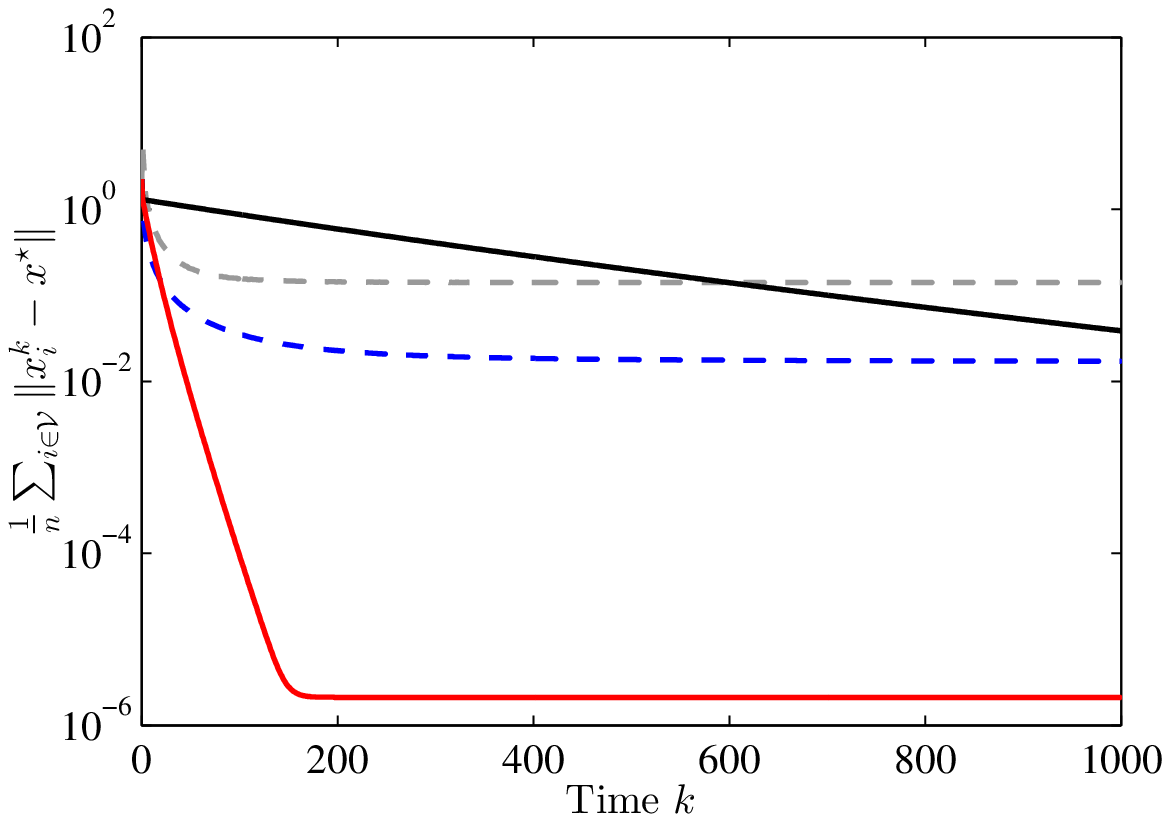}\label{fig:n500B10theta2}}
\subfigure[\xw{$n=50$, $B=10$, $0.2<\theta_i<0.4$}]{
\hspace{0.01\linewidth}
\includegraphics[width=0.3\linewidth, height=0.2\linewidth]{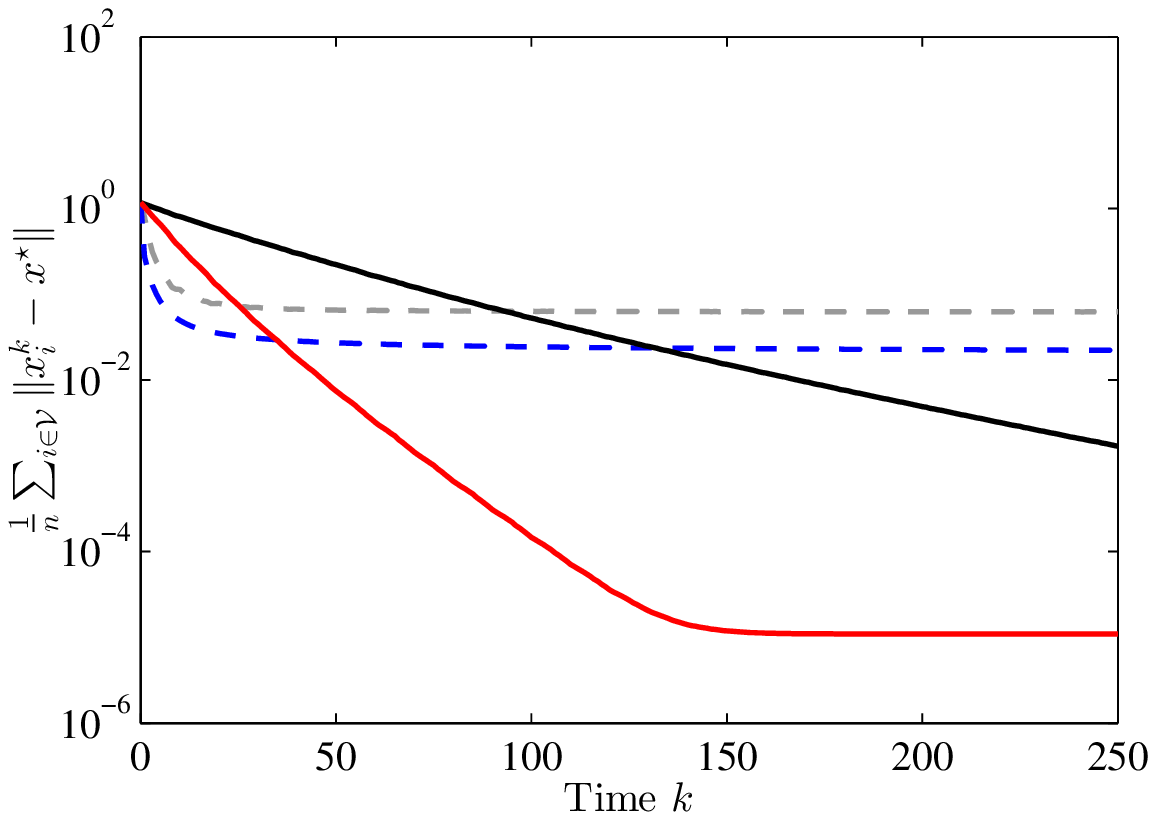}\label{fig:n50B10theta02}}
\vfill
\subfigure[$n=50$, $B=50$, $2<\theta_i<3$]{
\includegraphics[width=0.3\linewidth, height=0.2\linewidth]{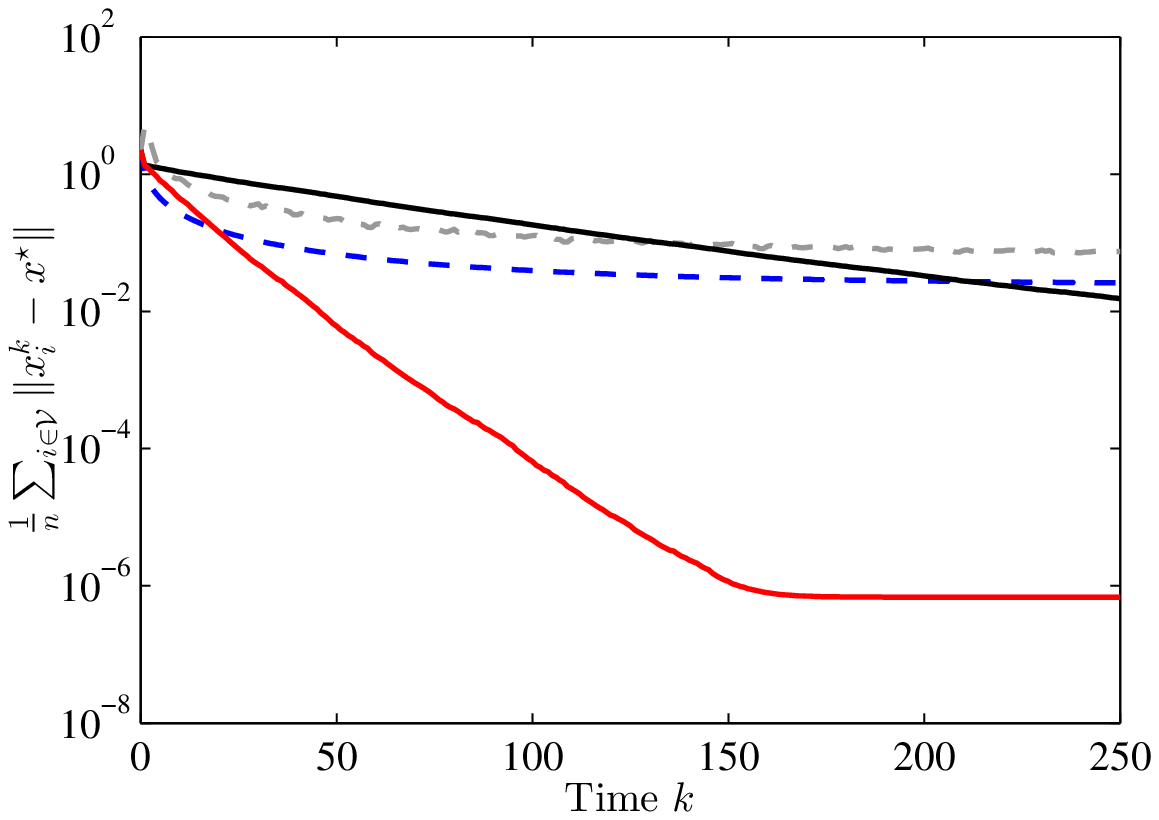}\label{fig:n50B50theta2}}
\hspace{0.01\linewidth}
\subfigure[$n=500$, $B=50$, $2<\theta_i<3$]{
\includegraphics[width=0.3\linewidth, height=0.2\linewidth]{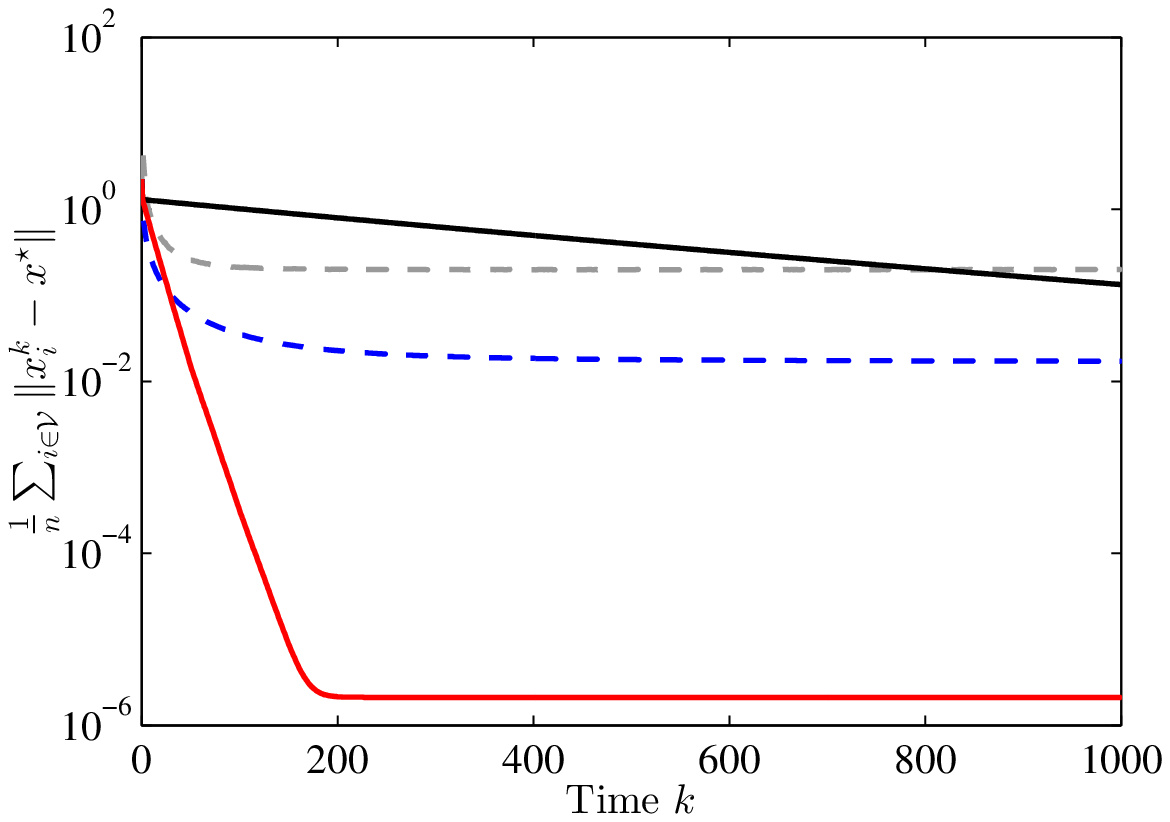}\label{fig:n500B50theta2}}
\hspace{0.01\linewidth}
\subfigure[\xw{$n=50$, $B=10$, $5<\theta_i<10$}]{
\includegraphics[width=0.3\linewidth, height=0.2\linewidth]{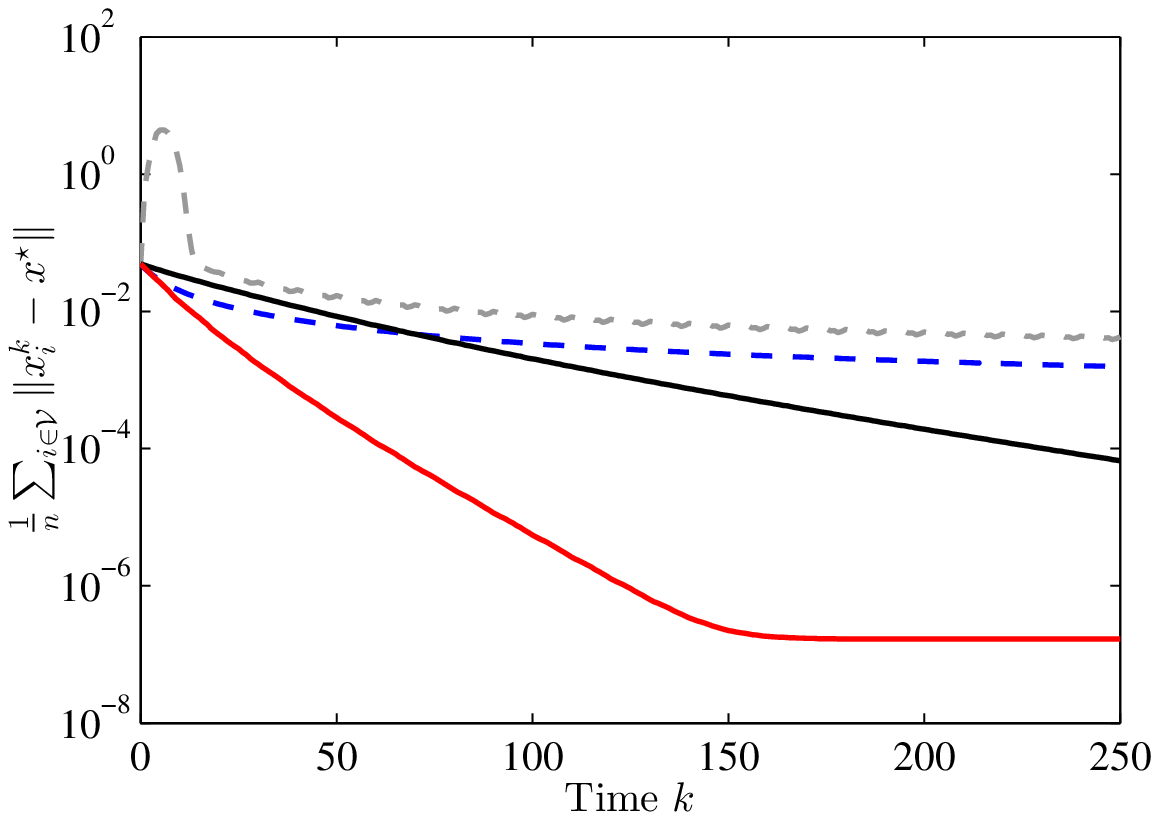}\label{fig:n50B10theta5}}
\caption{Primal errors in solving problem~\eqref{problem:testproblem} (The grey dashed, blue dashed, black solid, and red solid curves correspond to the consensus-based subgradient projection method, the proximal-minimization-based method, Algorithm~\ref{alg:weightgrad} with $H_{\mathcal{G}^k}$ in \eqref{eq:Laplacian}, and Algorithm~\ref{alg:weightgrad} with $H_{\mathcal{G}^k}$ in \eqref{eq:localweightmatrix}, respectively.).}
\label{fig:testfig}
\end{figure*}

We first compare the convergence performance of a consensus-based subgradient projection method \cite{Nedic10}, a proximal-minimization-based method \cite{Margellos16}, and Algorithm~\ref{alg:weightgrad} with $H_{\mathcal{G}^k}$ given by the graph Laplacian matrix \eqref{eq:Laplacian} and the Metropolis weight matrix \eqref{eq:localweightmatrix}, respectively, in solving constrained distributed optimization problems in the form of \eqref{eq:problem}. It has been proved that when each local constraint $X_i$ is compact, the consensus-based subgradient projection method and the proximal-minimization-based method, with diminishing step-sizes (e.g., $1/k$), asymptotically converge to an optimum over time-varying networks satisfying Assumption~\ref{asm:Bconnected} \cite{LinP16,Margellos16}. Thus, consider the following multi-agent $\ell_1$-regularization problem that often arises in machine learning:
\begin{align}
\begin{array}{ll}\underset{x\in\mathbb{R}^5}{\mbox{minimize}} & \sum_{i\in\mathcal{V}}(x^TA_ix+b_i^Tx+\frac{1}{n}\|x\|_1)\\ \operatorname{subject\,to} & x\in \bigcap_{i\in\mathcal{V}}\{x\in\mathbb{R}^5:p_i \leq x\leq q_i\},\end{array}\label{problem:testproblem}
\end{align}
where each $A_i\in\mathbb{R}^{5\times 5}$ is symmetric positive definite, $b_i\in\mathbb{R}^5$, and $p_i\leq x\leq q_i$ with $p_i,q_i\in\mathbb{R}^5$ means an elementwise inequality. In addition, for each $i\in\mathcal{V}$, the convexity parameter of its local objective is $\theta_i=\lambda_5^{\downarrow}(A_i)>0$.

For Algorithm~\ref{alg:weightgrad}, we adopt $\alpha^k=1/(Ln)$ for $H_{\mathcal{G}^k}$ in \eqref{eq:Laplacian} and $\alpha^k=1/2$ for $H_{\mathcal{G}^k}$ in \eqref{eq:localweightmatrix} to satisfy the step-size condition \eqref{eq:stepsize}. For the other two methods, we adopt the diminishing step-size $1/k$ and the local (unweighted) averaging operation as the consensus scheme to guarantee convergence. We also let the algorithms all start from the same initial primal iterate.

Figure~\ref{fig:testfig} presents the average primal errors produced by the aforementioned algorithms with different values of $n$, $B$ and $\theta_i$ $\forall i\in\mathcal{V}$. Observe that Algorithm~\ref{alg:weightgrad} with the Metropolis weight matrix \eqref{eq:localweightmatrix} outperforms the others in all six cases. Moreover, although at early stage the subgradient projection method and the proximal minimization method converge faster than Algorithm~\ref{alg:weightgrad} with the Laplacian weight matrix \eqref{eq:Laplacian}, their convergence gradually becomes much slower due to the diminishing nature of the step-size. By comparing Figure~\ref{fig:n50B10theta2} versus~\ref{fig:n50B50theta2} and Figure~\ref{fig:n500B10theta2} versus~\ref{fig:n500B50theta2}, we can see that smaller $B$ leads to faster convergence of Algorithm~\ref{alg:weightgrad}, which is consistent with our convergence analysis in Section~\ref{sec:convanal}, while the impact of $B$ on the subgradient projection method and the proximal minimization method is not apparent. \xw{Besides, Figure~\ref{fig:n50B10theta2} versus~\ref{fig:n500B10theta2} and Figure~\ref{fig:n50B50theta2} versus~\ref{fig:n500B50theta2} suggest that Algorithm~\ref{alg:weightgrad} with $H_{\mathcal{G}^k}$ in \eqref{eq:localweightmatrix} is more scalable to the network size $n$ than the others. Additionally, by comparing Figures~\ref{fig:n50B10theta02} and~\ref{fig:n50B10theta5} with Figure~\ref{fig:n50B10theta2}, it can be inferred that the larger the $\theta_i$'s are, the better Algorithm~\ref{alg:weightgrad} performs.}

\subsection{Unconstrained case}

\xw{In Section~\ref{ssec:dualBconnected}, we have compared Algorithm~\ref{alg:weightgrad} versus Subgradient-Push \cite{Nedic15}, Gradient-Push \cite{Nedic16b}, DIGing \cite{Nedic17}, and Push-DIGing \cite{Nedic17} in the theoretical aspects. Here, we compare, via simulation, their convergence performance in solving the following unconstrained quadratic program that satisfies all the assumptions in \cite{Nedic16b,Nedic17}:
\begin{align}
\textstyle{\operatorname{minimize}_{x\in\mathbb{R}^5}\sum_{i\in\mathcal{V}}}(x^TA_ix+b_i^Tx),\label{eq:unconstrainedprob}
\end{align}
where we let $\theta_i=\lambda_5^{\downarrow}(A_i)\in(2,3)$ $\forall i\in\mathcal{V}$ and $(n,B)=(50,10)$. For fair comparison, we assume there is no stochastic error in gradient evaluation for Gradient-Push. Then, Gradient-Push and Subgradient-Push have the same algorithmic form when the local objectives are differentiable, and below we omit Subgradient-Push.

Figure~\ref{fig:unconstrainedtheo} plots the evolution of the average primal error for Gradient-Push, DIGing, Push-DIGing, and Algorithm \ref{alg:weightgrad} with the Laplacian weight matrix \eqref{eq:Laplacian} and with the Metropolis weight matrix \eqref{eq:localweightmatrix}. We adopt the same step-sizes for Algorithm~\ref{alg:weightgrad} as in Section~\ref{ssec:constrained}. For the other three methods, we fine-tune the step-sizes while satisfying the step-size conditions in \cite{Nedic16b,Nedic17} that theoretically ensure their convergence rates. Observe that Gradient-Push, DIGing, and Push-DIGing almost stop making progress after a few iterations with a non-negligible primal error, while Algorithm~\ref{alg:weightgrad} achieves much better accuracy with the above two choices of $H_{\mathcal{G}^k}$. 

As all the convergence rate results in \cite{Nedic16b,Nedic17} and this paper are derived from worst-case analysis, the theoretical step-size conditions could be very conservative. Thus, in Figure~\ref{fig:unconstrainedbest} we empirically choose the step-sizes for these algorithms, whose values may violate the theoretical conditions but speed up convergence. After some tuning, we select the step-sizes to be $1/(nL)$, $1.7$, $0.15/k$, $0.05$, and $0.04$ for Algorithm \ref{alg:weightgrad} with $H_{\mathcal{G}^k}$ in \eqref{eq:Laplacian}, Algorithm~\ref{alg:weightgrad} with $H_{\mathcal{G}^k}$ in \eqref{eq:localweightmatrix}, Gradient-Push, DIGing, and Push-DIGing, respectively. Note that for Algorithm \ref{alg:weightgrad} with $H_{\mathcal{G}^k}$ in \eqref{eq:Laplacian}, the empirical step-size coincides with the theoretical one in Figure~\ref{fig:unconstrainedtheo}. By comparing Figure~\ref{fig:unconstrainedbest} with Figure~\ref{fig:unconstrainedtheo}, we can observe that with the above empirically-selected step-sizes, Gradient-Push slightly accelerates its convergence, DIGing and Push-DIGing exhibit prominently improved convergence performance, yet Algorithm~\ref{alg:weightgrad} with $H_{\mathcal{G}^k}$ in \eqref{eq:localweightmatrix} still performs best.}

\begin{figure}[tb]
\centering
\subfigure[Theoretically-selected step-sizes]{
\includegraphics[width=0.7\linewidth, height=0.4\linewidth]{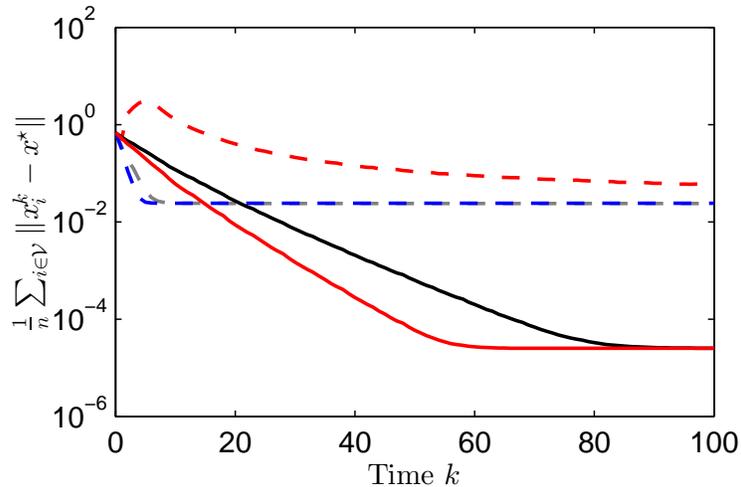}\label{fig:unconstrainedtheo}
}
\vfill
\subfigure[Empirically-selected step-sizes]{
\includegraphics[width=0.7\linewidth, height=0.4\linewidth]{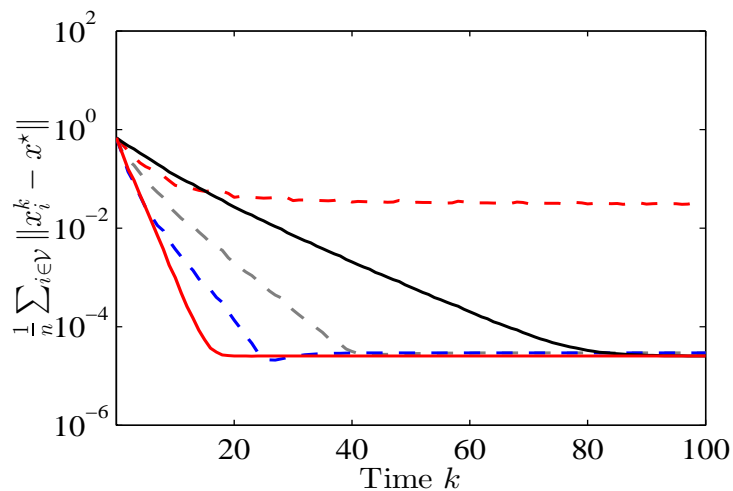}\label{fig:unconstrainedbest}
}
\caption{\xw{Primal errors in solving problem~\eqref{eq:unconstrainedprob} (The red dashed, grey dashed, blue dashed, black solid, and red solid curves correspond to Gradient-Push, DIGing, Push-DIGing, Algorithm~\ref{alg:weightgrad} with $H_{\mathcal{G}^k}$ in \eqref{eq:Laplacian}, and Algorithm~\ref{alg:weightgrad} with $H_{\mathcal{G}^k}$ in \eqref{eq:localweightmatrix}, respectively.).}}\label{fig:unconstrained}
\end{figure}

\section{Conclusion}\label{sec:conclusion}

We have constructed a family of distributed Fenchel dual gradient methods for solving multi-agent optimization problems with strongly convex local objectives and nonidentical local constraints over time-varying networks. The proposed algorithms have been proved to asymptotically converge to the optimal solution under a minimal connectivity condition, and have an $O(1/\sqrt{k})$ convergence rate under a standard connectivity condition. Simulation results have illustrated the competitive performance of the distributed Fenchel dual gradient methods by comparing them with related algorithms. \xw{In future, this work may be extended in a number of directions such as problems with general convex objective functions and networks with directed links.} 

\appendix
\section{Appendix}\label{sec:app}

\subsection{Proof of Proposition~\ref{pro:boundedlevelset}}\label{ssec:proofofpro:boundedlevelset}

Let $\mathbf{w}^{\star}=((w_1^\star)^T,\ldots,(w_n^\star)^T)^T$ be an optimal solution of problem \eqref{eq:dualprob}. Since Assumption~\ref{asm:problem}(b) assumes $\mathbf{0}_d\in\operatorname{int}\bigcap_{i\in\mathcal{V}}X_i$, there exists $r_c\in(0,\infty)$ such that $B(\mathbf{0}_d,r_c)\subseteq\bigcap_{i\in\mathcal{V}}X_i$. For each $i\in\mathcal{V}$, if $w_i^\star\neq\mathbf{0}_d$, let $x_i'=r_c\frac{w_i^{\star}}{\|w_i^{\star}\|}$; otherwise let $x_i'=\mathbf{0}_d$. Clearly, $x_i'\in B(\mathbf{0}_d,r_c)$. Consequently,
\begin{align*}
&D^\star=D(\mathbf{w}^{\star})=\sum_{i\in\mathcal{V}}\Bigl(\sup_{x_i\in X_i}(w_i^\star)^Tx_i-f_i(x_i)\Bigr)\displaybreak[0]\\
&\ge\sum_{i\in\mathcal{V}}\Bigl((w_i^\star)^Tx_i'-f_i(x_i')\Bigr)=r_c\sum_{i\in\mathcal{V}}\|w_i^\star\|-\sum_{i\in\mathcal{V}}f_i(x_i').
\end{align*}
This, along with $\|\mathbf{w}^\star\|\le\sum_{i\in\mathcal{V}}\|w_i^\star\|$ and $D^\star=-F^\star$, implies that $\|\mathbf{w}^\star\|\le \bigl((\sum_{i\in\mathcal{V}}f_i(x_i'))-F^\star\bigr)/r_c$. Note that $\sum_{i\in\mathcal{V}}f_i(x_i')\le\sum_{i\in\mathcal{V}}\max_{x_i\in B(\mathbf{0}_d,r_c)}f_i(x_i)$, where $F^\star\le\displaystyle{\sum_{i\in\mathcal{V}}\max_{x_i\in B(\mathbf{0}_d,r_c)}}f_i(x_i)<\infty$ because $B(\mathbf{0}_d,r_c)$ is compact. Therefore, \eqref{eq:dualoptimumbound} holds, which suggests that the optimal set of problem~\eqref{eq:dualprob} is compact. Then, due to the convexity of $D$ and $S^\bot$, the level sets $S_0(\mathbf{w})$ $\forall \mathbf{w}\in S^\bot$ are compact \cite[proposition 1.4.5]{Bertsekas09}.

\subsection{Proof of Lemma~\ref{lemma:funcvaldescent}}\label{ssec:proofoflemma:funcvaldescent}

For convenience, let $\mathbf{y}^k = (H_{\mathcal{G}^k}\otimes I_d)\nabla D(\mathbf{w}^k)$. Due to the Descent Lemma \cite{Bertsekas99} and \eqref{eq:weightgrad},
\begin{align}\label{eq:descentlemma}
D(\mathbf{w}^{k+1})-D(\mathbf{w}^k)&\leq\langle \nabla D(\mathbf{w}^k),\mathbf{w}^{k+1}-\mathbf{w}^k\rangle+(\mathbf{w}^{k+1}-\mathbf{w}^k)^T\frac{\Lambda_L\otimes I_d}{2}(\mathbf{w}^{k+1}-\mathbf{w}^k)\nonumber\\
&=-\alpha^k\langle \nabla D(\mathbf{w}^k), \mathbf{y}^k\rangle+(\alpha^k)^2(\mathbf{y}^k)^T\frac{\Lambda_L\otimes I_d}{2}\mathbf{y}^k.
\end{align}

Then, consider the following lemma. 
\begin{lemma}\label{lemma:psdmatrixbound}
Suppose $M, \bar{M}\in \mathbb{R}^{n\times n}$ are symmetric positive semidefinite and $M\preceq \bar{M}$. Then, for any $\mathbf{x}\in \mathbb{R}^{nd}$ and any $\mathbf{y}\in \mathcal{R}(M\otimes I_d)$,
\begin{align*}
\langle \mathbf{x}, (M\otimes I_d)\mathbf{x}\rangle \geq \langle (M\otimes I_d)\mathbf{x}, (\bar{M}^{\dag}\otimes I_d)(M\otimes I_d)\mathbf{x}\rangle.
\end{align*}
\end{lemma}

\begin{proof}
Let $\mathbf{x}\in \mathbb{R}^{nd}$. Then,
\begin{align}
\langle \mathbf{x}, (M\otimes I_d)\mathbf{x}\rangle-\langle (M\otimes I_d)\mathbf{x}, (\bar{M}^{\dag}\otimes I_d)(M\otimes I_d)\mathbf{x}\rangle= \mathbf{x}^T[(M-M\bar{M}^{\dag}M)\otimes I_d]\mathbf{x}.\label{eq:psdmatrixlarge}
\end{align}
In addition, by Schur complement condition, $M\succeq O_n$ and $\bar{M}\succeq M$ implies
\begin{align*}
    \left(
      \begin{array}{cc}
        M & M\\
        M & \bar{M}\\
      \end{array}
    \right) \succeq O_{2n}
\end{align*}
and the inequality above leads to $M-M\bar{M}^{\dag}M\succeq O_n$. Combining this with \eqref{eq:psdmatrixlarge}, the proof can be completed.
\end{proof}

From Lemma~\ref{lemma:psdmatrixbound}, $(\mathbf{y}^k)^T(\Lambda_L \otimes I_d)\mathbf{y}^k\!\leq\!\delta\langle \nabla D(\mathbf{w}^k), \mathbf{y}^k\rangle$. Combining this with \eqref{eq:descentlemma} leads to
\begin{align*}
D(\mathbf{w}^{k+1})-D(\mathbf{w}^k) \leq (\frac{(\alpha^k)^2\delta}{2}-\alpha^k)\langle \nabla D(\mathbf{w}^k), \mathbf{y}^k\rangle.
\end{align*}
This, along with \eqref{eq:stepsize}, completes the proof.

\subsection{Proof of Lemma~\ref{lemma:convergesgradient}}\label{ssec:proofoflemmaconvergesgradient}

We first consider the following optimization problem: For any $\mathcal{I}\subseteq \mathcal{V}$, $\mathcal{I}\neq\emptyset$ and any $c\in \mathbb{R}^d$,
\begin{align}
\begin{array}{ll}\underset{w_i\in\mathbb{R}^d\;\forall i\in\mathcal{I}}{\mbox{minimize}} & \sum_{i\in\mathcal{I}} d_i(w_i)\\ \operatorname{subject\,to} & \sum_{i\in \mathcal{I}} w_i=c.\end{array}
    \label{problem:infiniteconnectsubproblem}
\end{align}
Similar to problem~\eqref{eq:dualprob}, $w_i'$ $\forall i\in\mathcal{I}$ compose an optimum to \eqref{problem:infiniteconnectsubproblem} if and only if for any $i,j\in\mathcal{I}$, $\nabla d_i(w_i')=\nabla d_j(w_j')$ \cite[Lemma 3.1]{Lakshmanan08}, or equivalently, $\tilde{x}_i(w_i')=\tilde{x}_j(w_j')$. With the above setting, consider the following lemma.

\begin{lemma}\label{lemma:boundedconsensuserror}
Suppose Assumption~\ref{asm:problem} and the step-size condition \eqref{eq:stepsize} hold. Let $\mathbf{u},\mathbf{v}\in \mathbb{R}^{nd}$ be two feasible solutions of problem~\eqref{eq:dualprob} such that $u_i$ $\forall i\in\mathcal{I}$ and $v_i$ $\forall i\in\mathcal{I}$ are feasible to problem~\eqref{problem:infiniteconnectsubproblem}. Suppose $\|\tilde{x}_i(v_i)-\tilde{x}_j(v_j)\|\leq \epsilon'$ $\forall i,j\in \mathcal{I}$ for some $\epsilon'>0$, $\sum_{i\in \mathcal{I}} d_i(u_i)\leq \sum_{i\in \mathcal{I}} d_i(v_i)$, and $D(\mathbf{v})\leq D(\mathbf{w}^0)$, where $\mathbf{w}^0\in S^\bot$ is the initial dual iterate of Algorithm \ref{alg:weightgrad}. Then,
    \begin{align*}
        \|\tilde{x}_i(u_i)-\tilde{x}_j(u_j)\| \leq4\sqrt{LM_0(|\mathcal{I}|-1)\epsilon'},\quad\forall i,j\in\mathcal{I},
    \end{align*}
where $M_0$ is defined in \eqref{eq:M0}.
\end{lemma}

\begin{proof}
Let $\mathbf{w}'=(w_1'^T,\ldots,w_n'^T)^T\in\mathbb{R}^{nd}$ be such that $w_i'\in\mathbb{R}^d$ $\forall i\in\mathcal{I}$ compose an optimal solution to \eqref{problem:infiniteconnectsubproblem} and $w_j'=v_j$ $\forall j\notin \mathcal{I}$. Due to the convexity of each $d_i$ and \eqref{eq:dual_gradient},
\begin{align*}
\sum_{i\in \mathcal{I}} d_i(v_i)-\sum_{i\in \mathcal{I}} d_i(w_i')\leq \sum_{i\in \mathcal{I}} \langle \tilde{x}_i(v_i), v_i-w_i'\rangle.
\end{align*}
Let $\bar{x}_v:=\frac{1}{|\mathcal{I}|}\sum\limits_{i\in \mathcal{I}}\tilde{x}_i(v_i)$. Since $w_i'$ $\forall i\in\mathcal{I}$ and $v_i$ $\forall i\in\mathcal{I}$ are feasible to \eqref{problem:infiniteconnectsubproblem}, we have $\sum_{i\in\mathcal{I}}w_i'=\sum_{i\in\mathcal{I}}v_i$, which gives
\begin{align*}
\sum_{i\in \mathcal{I}} \langle \tilde{x}_i(v_i), v_i-w_i'\rangle&=\sum_{i\in \mathcal{I}} \langle \tilde{x}_i(v_i)-\bar{x}_v, v_i-w_i'\rangle\leq\sum_{i\in \mathcal{I}}\|\tilde{x}_i(v_i)-\bar{x}_v\|\cdot\|v_i-w_i'\|.
\end{align*}
Also note that for each $i\in\mathcal{I}$, $\|\tilde{x}_i(v_i)-\bar{x}_v\|=\frac{1}{|\mathcal{I}|}\|\sum_{j\in\mathcal{I}} (\tilde{x}_i(v_i)-\tilde{x}_j(v_j))\| \leq \frac{|\mathcal{I}|-1}{|\mathcal{I}|}\epsilon'$. Combining the above,
\begin{align}
\sum_{i\in \mathcal{I}} d_i(v_i)-\sum_{i\in \mathcal{I}} d_i(w_i')\leq \frac{|\mathcal{I}|-1}{|\mathcal{I}|}\epsilon'\sum_{i\in \mathcal{I}} \|v_i-w_i'\|\leq (|\mathcal{I}|-1)\epsilon'\sqrt{\sum_{i\in \mathcal{I}} \|v_i-w_i'\|^2}.\label{eq:sumdd<=I1epsilonvw}
\end{align}
Since $\sum_{i\in\mathcal{I}}d_i(w_i')\le\sum_{i\in\mathcal{I}}d_i(v_i)$ and $w_j'=v_j$ $\forall j\notin \mathcal{I}$, we have $D(\mathbf{w}')\leq D(\mathbf{v})\leq D(\mathbf{w}^0)$, implying that $\mathbf{w}',\mathbf{v}\in S_0(\mathbf{w}^0)$ and that for any optimum $\mathbf{w}^{\star}$ of problem~\eqref{eq:dualprob},
\begin{align*}
    \|\mathbf{w}'-\mathbf{v}\|\leq \|\mathbf{w}'-\mathbf{w}^{\star}\|+\|\mathbf{v}-\mathbf{w}^{\star}\|\leq 2M_0.
\end{align*}
This inequality and \eqref{eq:sumdd<=I1epsilonvw} together yield
\begin{align}
         &\sum_{i\in \mathcal{I}} d_i(v_i)-\sum_{i\in \mathcal{I}} d_i(w_i')\leq 2M_0(|\mathcal{I}|-1)\epsilon'.\label{eq:Lemma12ineq1}
\end{align}
Due to the optimality of $w_i'$ $\forall i\in\mathcal{I}$ with respect to \eqref{problem:infiniteconnectsubproblem}, we have $\nabla d_i(w_i')=\nabla d_j(w_j')$ $\forall i,j\in \mathcal{I}$. Also, because of the feasibility of $u_i$ $\forall i\in\mathcal{I}$, $\sum_{i\in \mathcal{I}}u_i=\sum_{i\in \mathcal{I}}w_i'$. Therefore, $\sum_{i\in \mathcal{I}}\langle \nabla d_i(w_i'), u_i-w_i'\rangle =0$. This, along with \eqref{eq:dual_gradient}, \eqref{eq:Lemma12ineq1}, and the inequality $d_i(u_i)-d_i(w_i')\geq\langle \nabla d_i(w_i'), u_i-w_i'\rangle + \frac{1}{2L}\|\nabla d_i(w_i')-\nabla d_i(u_i)\|^2$ \cite[Theorem 2.1.5]{Nesterov04}, implies
\begin{align*}
&\sum_{i\in \mathcal{I}}\|\tilde{x}_i(u_i)-\tilde{x}_i(w_i')\|^2\leq 2L\sum_{i\in \mathcal{I}}(d_i(u_i)-d_i(w_i'))\displaybreak[0]\\
&\leq 2L\sum_{i\in \mathcal{I}}(d_i(v_i)-d_i(w_i'))\leq 4LM_0(|\mathcal{I}|-1)\epsilon'.
\end{align*}
Hence, for any $i,j\in\mathcal{I}$, we have
$\|\tilde{x}_i(u_i)-\tilde{x}_j(u_j)\|\leq\|\tilde{x}_i(u_i)-\tilde{x}_i(w_i')\|+\|\tilde{x}_j(u_j)-\tilde{x}_j(w_j')\|\leq4\sqrt{LM_0(|\mathcal{I}|-1)\epsilon'}$, where the first inequality is from the optimality of $w_i'$ $\forall i\in\mathcal{I}$ and \eqref{eq:dual_gradient}.
\end{proof}

Next, we define the following: Arbitrarily pick $\epsilon>0$. Due to \eqref{eq:limmaxxx=0}, $\exists T_{\epsilon}\ge0$ such that
\begin{align}
\|x_i^k-x_j^k\|\leq \epsilon,\quad\forall\{i,j\}\in \mathcal{E}^k,\;\forall k\geq T_{\epsilon}.\label{eq:xx<=epsilon}
\end{align}
Then, for each $i\in\mathcal{V}$, let $\mathcal{C}_{i,\epsilon}^k=\emptyset$ $\forall k\in[0,T_\epsilon)$. For each $k\ge T_{\epsilon}$, let
\begin{align*}
\mathcal{C}_{i,\epsilon}^k=&\{i\}\cup\{j\in\mathcal{V}:\text{There exists a path between $i$ and $j$ }\text{in the graph }(\mathcal{V},\cup_{t=T_\epsilon}^k\mathcal{E}^t)\}\subseteq\mathcal{V}.
\end{align*}
For each $k\ge T_\epsilon$, observe that in the graph $(\mathcal{V},\cup_{t=T_\epsilon}^k\mathcal{E}^t)$, the subgraph induced by $\mathcal{C}_{i,\epsilon}^k$ is the largest connected component that contains node $i$.
Thus, for any two nodes $i$ and $j$, $i\neq j$, $\mathcal{C}_{i,\epsilon}^k$ and $\mathcal{C}_{j,\epsilon}^k$ are either identical or disjoint. Additionally, for every $s\in \mathcal{C}_{i,\epsilon}^{k+1}$, $\mathcal{C}_{s,\epsilon}^{k}$ is always contained in $\mathcal{C}_{i,\epsilon}^{k+1}$. This implies that the number of distinct sets in the collection $\{\mathcal{C}_{i,\epsilon}^k\}_{i\in\mathcal{V}}$ is non-increasing with $k$ over $[T_\epsilon,\infty)$. In particular, from each $k$ to $k+1$, $\mathcal{C}_{i,\epsilon}^{k+1}$ either equals $\mathcal{C}_{i,\epsilon}^{k}$ or is the union of $\mathcal{C}_{i,\epsilon}^{k}$ and some other $\mathcal{C}_{j,\epsilon}^{k}$'s that are disjoint from $\mathcal{C}_{i,\epsilon}^{k}$. Also due to Assumption~\ref{asm:infiniteconnect}, there exists $K_\epsilon\in[T_\epsilon,\infty)$ such that $\mathcal{C}_{i,\epsilon}^k=\mathcal{V}$ $\forall i\in\mathcal{V}$ $\forall k\ge K_\epsilon$.
By means of the $\mathcal{C}_{i,\epsilon}^k$'s and Lemma~\ref{lemma:boundedconsensuserror}, below we show that $\forall i\in\mathcal{V}$, $\forall k\ge T_\epsilon$,
\begin{align}
\max_{j,\ell\in \mathcal{C}_{i,\epsilon}^k}\|x_j^k-x_\ell^k\|\le \Phi_i^k(\epsilon).\label{eq:maxxx<=Phi}
\end{align}
Here, $\Phi_i^k(\epsilon)$ $\forall i\in\mathcal{V}$ $\forall k\ge T_\epsilon$ are defined recursively as follows:
Initially at $k=T_{\epsilon}$, $\Phi_i^{k}(\epsilon)=(|\mathcal{C}_{i,\epsilon}^k|-1)\epsilon$. At each subsequent $k\ge T_\epsilon+1$,
\begin{align*}
\Phi_i^k(\epsilon)=\begin{cases}4\sqrt{LM_0(|\mathcal{C}_{i,\epsilon}^k|-1)\Phi_i^{t^k}(\epsilon)},&\text{if $\mathcal{C}_{i,\epsilon}^k=\mathcal{C}_{i,\epsilon}^{k-1}$,}\\
(1+2L\bar{\alpha}\bar{h}n)|\mathcal{C}_{i,\epsilon}^k|\epsilon
+\sum_{s\in \mathcal{C}_{i,\epsilon}^k}\Phi_s^{k-1}(\epsilon), &\text{\;otherwise,}
\end{cases}
\end{align*}
where $t^k:= \max\{t\in[T_\epsilon,k]:\mathcal{C}_{i,\epsilon}^t\neq\mathcal{C}_{i,\epsilon}^{t-1}\}$. Note that $\mathcal{C}_{i,\epsilon}^k=\mathcal{C}_{i,\epsilon}^t$ $\forall t\in[t^k,k]$.

We prove \eqref{eq:maxxx<=Phi} by induction. At time $k=T_\epsilon$, for each $i\in\mathcal{V}$, if $|\mathcal{C}_{i,\epsilon}^k|=1$, then $\max_{j,\ell\in \mathcal{C}_{i,\epsilon}^k}\|x_j^k-x_\ell^k\|=\Phi_i^k(\epsilon)=0$, i.e., \eqref{eq:maxxx<=Phi} is satisfied; otherwise for any $j,\ell\in \mathcal{C}_{i,\epsilon}^k$, $j\neq\ell$, there exists a path of length at most $|\mathcal{C}_{i,\epsilon}^k|-1$ connecting $j$ and $\ell$. It follows from \eqref{eq:xx<=epsilon} that $\|x_j^k-x_\ell^k\|\le(|\mathcal{C}_{i,\epsilon}^k|-1)\epsilon=\Phi_i^k(\epsilon)$, i.e., \eqref{eq:maxxx<=Phi} also holds. Next, suppose $\max_{j,\ell\in \mathcal{C}_{i,\epsilon}^{t}}\|x_j^{t}-x_\ell^{t}\|\le \Phi_i^{t}(\epsilon)$ $\forall i\in\mathcal{V}$ $\forall t\in[T_\epsilon,k-1]$ for some $k\ge T_\epsilon+1$. For each $i\in\mathcal{V}$, to show that \eqref{eq:maxxx<=Phi} holds, consider the following two cases.

\emph{Case i}: $\mathcal{C}_{i,\epsilon}^k=\mathcal{C}_{i,\epsilon}^{k-1}$. In this case, we have $T_\epsilon\le t^k\le k-1$. Also, $\forall t\in[t^k+1,k]$, $\forall j\in\mathcal{C}_{i,\epsilon}^{t-1}$, we have $\mathcal{N}_j^t\subseteq\mathcal{C}_{i,\epsilon}^{t-1}=\mathcal{C}_{i,\epsilon}^k$ . Hence, using the same arguments as the proofs of Proposition~\ref{pro:dualfeasible} and Lemma~\ref{lemma:funcvaldescent}, it can be shown that $\sum_{s\in\mathcal{C}_{i,\epsilon}^k}w_s^k=\sum_{s\in\mathcal{C}_{i,\epsilon}^k}w_s^{k-1}=\cdots=\sum_{s\in\mathcal{C}_{i,\epsilon}^k}w_s^{t^k}$ and that $\sum_{s\in\mathcal{C}_{i,\epsilon}^k}d_s(w_s^k)\le\sum_{s\in\mathcal{C}_{i,\epsilon}^k}d_s(w_s^{k-1})\le\cdots\le\sum_{s\in\mathcal{C}_{i,\epsilon}^k}d_s(w_s^{t^k})$. Let $\mathcal{I}=\mathcal{C}_{i,\epsilon}^k$ and $c=\sum_{s\in\mathcal{C}_{i,\epsilon}^k}w_s^{t^k}$ in problem~\eqref{problem:infiniteconnectsubproblem}. It then follows from Lemma~\ref{lemma:funcvaldescent} and Lemma \ref{lemma:boundedconsensuserror} with
$\epsilon'=\Phi_i^{t^k}(\epsilon)$, $\mathbf{u}=\mathbf{w}^k$, and $\mathbf{v}=\mathbf{w}^{t^k}$ that \eqref{eq:maxxx<=Phi} holds.

\emph{Case ii}: $\mathcal{C}_{i,\epsilon}^k\neq\mathcal{C}_{i,\epsilon}^{k-1}$. Pick any $j,\ell\in\mathcal{C}_{i,\epsilon}^k$, $j\neq\ell$ and consider the following two subcases.

\emph{Subcase ii(a)}: $\mathcal{C}_{j,\epsilon}^{k-1}=\mathcal{C}_{\ell,\epsilon}^{k-1}$. Then,
$\|x_j^k-x_\ell^k\|\le\|x_j^k-x_j^{k-1}\|+\|x_j^{k-1}-x_\ell^{k-1}\|+\|x_\ell^{k-1}-x_\ell^k\|\le\|x_j^k-x_j^{k-1}\|+\|x_\ell^k-x_\ell^{k-1}\|+\Phi_j^{k-1}(\epsilon)$.
Also, from \eqref{eq:dual_gradient}, Proposition~\ref{pro:Lipschitz}, \eqref{eq:weightgrad}, and \eqref{eq:xx<=epsilon}, we have
\begin{align*}
&\|x_p^k-x_p^{k-1}\|\leq L_p\|w_p^k-w_p^{k-1}\|\le L\bar{\alpha}\|\sum_{q\in\mathcal{N}_p^{k-1}}h_{pq}^{k-1}(x_p^{k-1}-x_q^{k-1})\|\nonumber\displaybreak[0]\\
&\le L\bar{\alpha}\bar{h}\sum_{q\in\mathcal{N}_p^{k-1}}\|x_p^{k-1}-x_q^{k-1}\|\le L\bar{\alpha}\bar{h}n\epsilon,\quad\forall p\in\mathcal{V}.
\end{align*}
Consequently, $\|x_j^k-x_\ell^k\|\le2L\bar{\alpha}\bar{h}n\epsilon+\Phi_j^{k-1}(\epsilon)$.

\emph{Subcase ii(b)}: $\mathcal{C}_{j,\epsilon}^{k-1}\cap\mathcal{C}_{\ell,\epsilon}^{k-1}=\emptyset$. Then, there exists a path from $j$ to $\ell$ belonging to the subgraph induced in the graph $(\mathcal{V},\cup_{t=T_\epsilon}^k\mathcal{E}^t)$ by $\mathcal{C}_{i,\epsilon}^k$. Along the path are nodes $p_1=j,s_1,p_2,s_2,\ldots,p_\tau,s_\tau=\ell$ such that (1) $\mathcal{C}_{p_r,\epsilon}^{k-1}=\mathcal{C}_{s_r,\epsilon}^{k-1}$ $\forall r=1,\ldots,\tau$; (2) $\mathcal{C}_{p_r,\epsilon}^{k-1}$ $\forall r\in\{1,\ldots,\tau\}$ are disjoint from each other; and (3) $\{s_r,p_{r+1}\}\in\mathcal{E}^k$ $\forall r\in\{1,\ldots,\tau-1\}$. Here, $\tau\in\{2,\ldots,|\mathcal{C}_{i,\epsilon}^k|\}$ is an integer whose value is no more than the number of distinct sets in the collection $\{\mathcal{C}_{s,\epsilon}^{k-1}\}_{s\in\mathcal{C}_{i,\epsilon}^k}$. Hence, $\|x_j^k-x_\ell^k\|\le\|x_{p_1}^k-x_{s_1}^k\|+\sum_{r=1}^{\tau-1}\bigl(\|x_{s_r}^k-x_{p_{r+1}}^k\|+\|x_{p_{r+1}}^k-x_{s_{r+1}}^k\|\bigr)$. For each $r=1,\ldots,\tau$, since $p_r,s_r\in \mathcal{C}_{p_r,\epsilon}^{k-1}$, we obtain from \emph{Subcase ii(a)} that $\|x_{p_r}^k-x_{s_r}^k\|\le2L\bar{\alpha}\bar{h}n\epsilon+\Phi_{p_r}^{k-1}(\epsilon)$. It then follows from \eqref{eq:xx<=epsilon} that $\|x_j^k-x_\ell^k\|\le(\tau-1)\epsilon+2\tau L\bar{\alpha}\bar{h}n\epsilon+\sum_{r=1}^{\tau}\Phi_{p_r}^{k-1}(\epsilon)\le(1+2L\bar{\alpha}\bar{h}n)|\mathcal{C}_{i,\epsilon}^k|\epsilon
+\sum_{s\in \mathcal{C}_{i,\epsilon}^k}\Phi_s^{k-1}(\epsilon)$.

Combining the above two subcases, we obtain \eqref{eq:maxxx<=Phi}. This completes the proof of \eqref{eq:maxxx<=Phi} for all $i\in\mathcal{V}$ and all $k\ge T_\epsilon$. Further, notice that for each $i\in\mathcal{V}$, $\Phi_i^k(\epsilon)$ is updated only if either $\mathcal{C}_{i,\epsilon}^k$ or $\mathcal{C}_{i,\epsilon}^{k-1}$ is changed. Also note that $\mathcal{C}_{i,\epsilon}^k$ can be expanded at most $n$ times and remains unchanged since time $K_\epsilon$. Therefore, for any $k\ge K_\epsilon+1$,
\begin{align*}
&\max_{i,j\in\mathcal{V}}\|x_i^k-x_j^k\|=\max_{i\in\mathcal{V}}\max_{j,\ell\in \mathcal{C}_{i,\epsilon}^k}\|x_j^k-x_\ell^k\|\le\max_{i\in\mathcal{V}}\Phi_i^k(\epsilon)\le O(\epsilon^{1/2^n}),
\end{align*}
which implies $\max_{i,j\in\mathcal{V}}\|x_i^k-x_j^k\|\rightarrow0$ as $k\rightarrow\infty$.

\subsection{Proof of Theorem~\ref{thm:asymconv}}\label{ssec:proofofthm:asymconv}

Let $\mathbf{w}^\star$ be an optimal solution to the dual problem~\eqref{eq:dualprob}. Due to the convexity of $D$, \eqref{eq:dual_gradient}, and Proposition~\ref{pro:dualfeasible},
\begin{align*}
&D(\mathbf{w}^k)-D^{\star}\leq \langle \nabla D(\mathbf{w}^k), \mathbf{w}^k-\mathbf{w}^{\star}\rangle=\langle \mathbf{x}^k, \mathbf{w}^k-\mathbf{w}^{\star}\rangle\displaybreak[0]\\
&\quad\leq \|P_{S^{\bot}}(\mathbf{x}^k)\|\cdot\| \mathbf{w}^k-\mathbf{w}^{\star}\|\displaybreak[0]\leq M_0\|P_{S^{\bot}}(\mathbf{x}^k)\|,
\end{align*}
where $M_0$ is defined in \eqref{eq:M0}. As $k\rightarrow \infty$, we have shown in the paragraph below Lemma \ref{lemma:convergesgradient} that $\|P_{S^{\bot}}(\mathbf{x}^k)\|\rightarrow 0$. This, along with the above inequality, implies $D(\mathbf{w}^k)\rightarrow D^{\star}$. In addition, since Assumption~\ref{asm:problem} guarantees zero duality gap, we have $F(\mathbf{x}^k)\rightarrow F^{\star}$. Finally, for any $\mathbf{w}\in S^\bot$, due to Corollary~\ref{cor:Lipschitz}, \cite[Theorem 2.1.5]{Nesterov04}, and \eqref{eq:dualgradfull},
\begin{align}
D(\mathbf{w})-D^{\star}\geq\langle \nabla D(\mathbf{w}^{\star}), \mathbf{w}-\mathbf{w}^{\star} \rangle+\frac{1}{2L}\|\nabla D(\mathbf{w})-\nabla D(\mathbf{w}^{\star})\|^2=\frac{1}{2L}\|\tilde{\mathbf{x}}(\mathbf{w})-\mathbf{x}^{\star}\|^2,\label{eq:DD>=12Lxx}
\end{align}
where the last equality is because $\nabla D(\mathbf{w}^{\star})=\mathbf{x}^{\star}\in S$ and $\mathbf{w}, \mathbf{w}^{\star}\in S^{\bot}$. Thus, because $\lim_{k\rightarrow\infty}D(\mathbf{w}^k)-D^{\star}=0$ and $L>0$, $\|\mathbf{x}^k-\mathbf{x}^{\star}\|^2\rightarrow 0$ as $k\rightarrow\infty$.

\subsection{Proof of Lemma~\ref{lemma:finallemma}}\label{ssec:proofoflemma:finallemma}

Let $k\in\{0,B,2B,\ldots\}$. For each $\{i,j\}\in \tilde{\mathcal{E}}^k$, let $t_{\{i,j\}}^k\in\{k,\ldots,k+B-1\}$ be such that $\{i,j\}\in\mathcal{E}^{t_{\{i,j\}}^k}$. Then, note from Proposition~\ref{pro:Lipschitz} that
\begin{align*}
&\|\nabla d_i(w_i^k)-\nabla d_i(w_i^{t_{\{i,j\}}^k})\|^2 =\|\sum_{t=k}^{t_{\{i,j\}}^k-1}(\nabla d_i(w_i^{t+1})-\nabla d_i(w_i^t))\|^2\displaybreak[0]\\
&\le B\sum_{t=k}^{k+B-1}\|\nabla d_i(w_i^{t+1})-\nabla d_i(w_i^t)\|^2\le L_i^2B\sum_{t=k}^{k+B-1}\|w_i^{t+1}-w_i^t\|^2.
\end{align*}
Thus,
\begin{align*}
&\sum_{\{i,j\}\in\tilde{\mathcal{E}}^k}(\|\nabla d_i(w_i^k)-\nabla d_i(w_i^{t_{\{i,j\}}^k})\|^2+\|\nabla d_j(w_j^{t_{\{i,j\}}^k})-\nabla d_j(w_j^k)\|^2)\displaybreak[0]\\
\le& B\sum_{\{i,j\}\in\tilde{\mathcal{E}}^k}\sum_{t=k}^{k+B-1} (L_i^2\|w_i^{t+1}-w_i^t\|^2+L_j^2\|w_j^{t+1}-w_j^t\|^2)\displaybreak[0]\\
\le& B\bar{\varpi}\sum_{t=k}^{k+B-1}\sum_{i\in\mathcal{V}}L_i^2\|w_i^{t+1}-w_i^t\|^2\displaybreak[0]\\
\le& B\bar{\varpi}\bar{\alpha}^2\sum_{t=k}^{k+B-1}\langle\nabla D(\mathbf{w}^t), ((H_{\mathcal{G}^t}\Lambda_L^2H_{\mathcal{G}^t})\otimes I_d)\nabla D(\mathbf{w}^t)\rangle.
\end{align*}
Note that $H_{\mathcal{G}^t}\Lambda_L^2H_{\mathcal{G}^t}\preceq LH_{\mathcal{G}^t}\Lambda_LH_{\mathcal{G}^t}$. Also, from \eqref{eq:delta} and Lemma \ref{lemma:psdmatrixbound}, $H_{\mathcal{G}^t}\Lambda_LH_{\mathcal{G}^t} \preceq \delta H_{\mathcal{G}^t}$. Hence,
\begin{align}
&\sum_{\{i,j\}\in\tilde{\mathcal{E}}^k}(\|\nabla\! d_i(w_i^k)-\nabla d_i(w_i^{t_{\{i,j\}}^k})\|^2+\|\nabla d_j(w_j^{t_{\{i,j\}}^k})-\nabla\! d_j(w_j^k)\|^2)\nonumber\displaybreak[0]\\
&\le B\bar{\varpi}\bar{\alpha}^2\delta L\sum_{t=k}^{k+B-1}\nabla D(\mathbf{w}^t)^T(H_{\mathcal{G}}^t\otimes I_d)\nabla D(\mathbf{w}^t).\label{eq:sumndndndnd}
\end{align}
In addition,
\begin{align}
\sum_{\{i,j\}\in\tilde{\mathcal{E}}^k}\|\nabla d_i(w_i^{t_{\{i,j\}}^k})-\nabla d_j(w_j^{t_{\{i,j\}}^k})\|^2&\le\frac{1}{\underline{h}}\sum_{t=k}^{k+B-1}\sum_{\{i,j\}\in\mathcal{E}^t}h_{ij}^k\|\nabla d_i(w_i^t)-\nabla d_j(w_j^t)\|^2\nonumber\displaybreak[0]\\
&\le\frac{1}{\underline{h}}\sum_{t=k}^{k+B-1}\nabla D(\mathbf{w}^t)^T(H_{\mathcal{G}}^t\otimes I_d)\nabla D(\mathbf{w}^t).\label{eq:sumndnd<=1hsumnDHnD}
\end{align}
It follows from \eqref{eq:sumndndndnd} and \eqref{eq:sumndnd<=1hsumnDHnD} that
\begin{align*}
&\nabla D(\mathbf{w}^k)^T(L_{\tilde{\mathcal{G}}^k}\otimes I_d)\nabla D(\mathbf{w}^k)=\sum_{\{i,j\}\in\tilde{\mathcal{E}}^k}\|\nabla d_i(w_i^k)-\nabla d_j(w_j^k)\|^2\displaybreak[0]\\
&\le\!3\!\!\!\!\sum_{\{i,j\}\in\tilde{\mathcal{E}}^k}\!\!\!\!\!(\|\nabla d_i(w_i^k)\!-\!\nabla d_i(w_i^{t_{\{i,j\}}^k})\|^2\!+\!\|\nabla d_j(w_j^{t_{\{i,j\}}^k})\!-\!\nabla d_j(w_j^k)\|^2\!+\!\|\nabla d_i(w_i^{t_{\{i,j\}}^k})\!-\!\nabla d_j(w_j^{t_{\{i,j\}}^k})\|^2)\displaybreak[0]\\
&\le\eta\sum_{t=k}^{k+B-1}\nabla D(\mathbf{w}^t)^T(H_{\mathcal{G}}^t\otimes I_d)\nabla D(\mathbf{w}^t).
\end{align*}

\subsection{Proof of Theorem~\ref{theorem:dualconvrate}}\label{ssec:proofoftheorem:dualconvrate}

Let $k\ge0$. By Lemmas~\ref{lemma:funcvaldescent} and~\ref{lemma:finallemma},
\begin{align}
&\bigl(D(\mathbf{w}^{(k+1)B})-D^\star\bigr)-\bigl(D(\mathbf{w}^{kB})-D^\star\bigr)=\sum_{t=kB}^{(k+1)B-1}(D(\mathbf{w}^{t+1})\!-\!D(\mathbf{w}^t))\displaybreak[0]\nonumber\\
&\le-\rho\sum_{t=kB}^{(k+1)B-1}\nabla D(\mathbf{w}^t)^T(H_{\mathcal{G}^t}\otimes I_d)\nabla D(\mathbf{w}^t)\le-\frac{\rho}{\eta}\nabla D(\mathbf{w}^{kB})^T(L_{\tilde{\mathcal{G}}^{kB}}\otimes I_d)\nabla D(\mathbf{w}^{kB})\displaybreak[0]\nonumber\\
&\le-\frac{\rho\underline{\lambda}}{\eta}\|P_{S^\bot}(\nabla D(\mathbf{w}^{kB}))\|^2,\label{eq:DDDD<=rhonormB}
\end{align}
where the last inequality is because $\tilde{\mathcal{G}}^{kB}$ is connected and thus $\operatorname{Null}(L_{\tilde{\mathcal{G}}^{kB}}\otimes I_d)=S$. Also, since $\tilde{\mathcal{G}}^{tB}$ $\forall t=0,1,\ldots$ are connected, we have $\underline{\lambda}>0$. From Proposition~\ref{pro:dualfeasible}, we know that $\mathbf{w}^{kB}\in S^\bot$. Also, for any optimal solution $\mathbf{w}^\star$ to \eqref{eq:dualprob}, because $\mathbf{w}^\star\in S^\bot$, we have $\mathbf{w}^{kB}-\mathbf{w}^{\star}\in S^\bot$. Then,
\begin{align*}
D(\mathbf{w}^{kB})-D^\star\leq&\langle \nabla D(\mathbf{w}^{kB}), \mathbf{w}^{kB}-\mathbf{w}^{\star}\rangle=\langle P_{S^\bot}(\nabla D(\mathbf{w}^{kB})), \mathbf{w}^{kB}-\mathbf{w}^{\star}\rangle\displaybreak[0]\\
\leq &\|P_{S^\bot}(\nabla D(\mathbf{w}^{kB}))\|\cdot\|\mathbf{w}^{kB}-\mathbf{w}^{\star}\|.
\end{align*}
This, along with \eqref{eq:DDDD<=rhonormB}, gives
\begin{align*}
\bigl(D(\mathbf{w}^{(k+1)B})-D^\star\bigr)-\bigl(D(\mathbf{w}^{kB})-D^\star\bigr)\le-\rho\underline{\lambda}\bigl(D(\mathbf{w}^{kB})-D^\star\bigr)^2/(\eta\!\!\!\min_{\mathbf{w}^\star\in S^\bot:D(\mathbf{w}^\star)=D^\star}\!\!\!\!\!\!\!\!\!\|\mathbf{w}^{kB}-\mathbf{w}^{\star}\|^2). 
\end{align*}
Finally, using Lemma 6 in \cite[Sec. 2.2.1]{Polyak87}, we obtain
\begin{align*}
D(\mathbf{w}^{kB})-D^{\star}\le&\frac{D(\mathbf{w}^0)-D^\star}{1+\frac{\rho\underline{\lambda}(D(\mathbf{w}^0)-D^\star)}{\eta}\sum\limits_{t=0}^{k-1}(\min\limits_{\mathbf{w}^\star\in S^\bot:D(\mathbf{w}^\star)=D^\star}\|\mathbf{w}^{tB}-\mathbf{w}^{\star}\|^2)^{-1}}\displaybreak[0]\\
\le&\frac{D(\mathbf{w}^0)-D^\star}{1+\rho\underline{\lambda}(D(\mathbf{w}^0)-D^\star)k/(\eta \tilde{M}_k^2)}.
\end{align*}
Note that the above inequality is equivalent to \eqref{eq:theorem1result} since $(D(\mathbf{w}^k))_{k=0}^\infty$ is non-increasing.

\subsection{Proof of Theorem~\ref{thm:primalconvrate}}\label{ssec:proofofthm:primalconvrate}

Let $\mathbf{w}\in S^\bot$. Note that $\|P_{S^\bot}(\tilde{\mathbf{x}}(\mathbf{w}))\|=\|\tilde{\mathbf{x}}(\mathbf{w})-P_S(\tilde{\mathbf{x}}(\mathbf{w}))\|\le\|\tilde{\mathbf{x}}(\mathbf{w})-\mathbf{x}^\star\|$. Thus, from \eqref{eq:DD>=12Lxx},
\begin{align}
\|P_{S^\bot}(\tilde{\mathbf{x}}(\mathbf{w}))\|\le\|\tilde{\mathbf{x}}(\mathbf{w})-\mathbf{x}^{\star}\|\leq\sqrt{2L(D(\mathbf{w})-D^\star)}.\label{eq:lemmaconsensus}
\end{align}
Also note that
\begin{align*}
&F(\tilde{\mathbf{x}}(\mathbf{w}))-F^{\star} = \langle\mathbf{w},\tilde{\mathbf{x}}(\mathbf{w})\rangle-D(\mathbf{w})+D^\star\leq \langle\mathbf{w},\tilde{\mathbf{x}}(\mathbf{w})\rangle=\langle\mathbf{w},P_{S^\bot}(\tilde{\mathbf{x}}(\mathbf{w}))\rangle.
\end{align*}
On the other hand, for any dual optimum $\mathbf{w}^\star\in S^\bot$, we have $-F^\star=D^\star\ge\langle\mathbf{w}^\star,\tilde{\mathbf{x}}(\mathbf{w})\rangle-F(\tilde{\mathbf{x}}(\mathbf{w}))$, which leads to
\begin{align*}
F(\tilde{\mathbf{x}}(\mathbf{w}))-F^\star\ge\langle\mathbf{w}^\star,P_{S^\bot}(\tilde{\mathbf{x}}(\mathbf{w}))\rangle.
\end{align*}
As a result,
\begin{align}
&-\|\mathbf{w}^\star\|\cdot\|P_{S^\bot}(\tilde{\mathbf{x}}(\mathbf{w}))\|\le F(\tilde{\mathbf{x}}(\mathbf{w}))-F^\star\le\|\mathbf{w}\|\cdot\|P_{S^\bot}(\tilde{\mathbf{x}}(\mathbf{w}))\|.\label{eq:lemmafuncval}
\end{align}
Combining \eqref{eq:lemmaconsensus} and \eqref{eq:lemmafuncval} with Proposition~\ref{pro:dualfeasible} and Theorem~\ref{theorem:dualconvrate} completes the proof.

\bibliographystyle{IEEEtran}
\bibliography{reference}
\end{document}